\documentclass[12pt]{amsart}       
\usepackage{geometry}
\usepackage{extarrows}
\usepackage{amssymb}
\usepackage{eucal}
\usepackage{bbm}
\usepackage{graphicx}
\usepackage{amsmath}
\usepackage{esint}
\usepackage{amscd}
\usepackage[all]{xy}           
\usepackage{amsfonts,latexsym}
\usepackage{xspace}
\usepackage{epsfig}
\usepackage{float}
\usepackage{color}
\usepackage{shuffle}
\usepackage{fancybox}
\usepackage{colordvi}
\usepackage{multicol}
\usepackage{mathrsfs}
\usepackage{ifpdf}
\ifpdf
  \usepackage[colorlinks,final,backref=page,hyperindex]{hyperref}
\else
  \usepackage[colorlinks,final,backref=page,hyperindex,hypertex]{hyperref}
\fi
\usepackage[active]{srcltx} 
\usepackage{tikz}
\DeclareMathAlphabet{\mathcal}{OMS}{cmsy}{m}{n}

\newtheorem{theorem}{Theorem}[section]
\newtheorem{proposition}[theorem]{Proposition}
\newtheorem{lemma}[theorem]{Lemma}
\newtheorem{coro}[theorem]{Corollary}
\newtheorem{assumption}[theorem]{Assumption}
\newtheorem{prop-def}{Proposition-Definition}[section]
\newtheorem{coro-def}{Corollary-Definition}[section]

\theoremstyle{definition}
\newtheorem{definition}[theorem]{Definition}
\newtheorem{remark}[theorem]{Remark}

\newcommand{\nc}{\newcommand}
\nc{\tred}[1]{\textcolor{red}{#1}}
\nc{\tblue}[1]{\textcolor{blue}{#1}}
\nc{\tgreen}[1]{\textcolor{green}{#1}}
\nc{\tpurple}[1]{\textcolor{purple}{#1}}
\nc{\btred}[1]{\textcolor{red}{\bf #1}}
\nc{\btblue}[1]{\textcolor{blue}{\bf #1}}
\nc{\btgreen}[1]{\textcolor{green}{\bf #1}}
\nc{\btpurple}[1]{\textcolor{purple}{\bf #1}}
\nc{\NN}{{\mathbb N}}
\nc{\ncsha}{{\mbox{\cyr X}^{\mathrm NC}}} \nc{\ncshao}{{\mbox{\cyr
X}^{\mathrm NC}_0}}

\newcommand{\delete}[1]{}

\nc{\mlabel}[1]{\label{#1}}
\nc{\mcite}[1]{\cite{#1}}
\nc{\mref}[1]{\ref{#1}}
\nc{\meqref}[1]{\eqref{#1}}
\nc{\mbibitem}[1]{\bibitem{#1}}

\delete{
\nc{\mlabel}[1]{\label{#1}{\hfill \hspace{1cm}{\bf{{\ }\hfill(#1)}}}}
\nc{\mcite}[1]{\cite{#1}{{\bf{{\ }(#1)}}}}
\nc{\mref}[1]{\ref{#1}{{\bf{{\ }(#1)}}}}
\nc{\meqref}[1]{\eqref{#1}{{\bf{{\ }(#1)}}}}
\nc{\mbibitem}[1]{\bibitem[\bf #1]{#1}}
}
\nc{\sha}{{\mbox{\cyr X}}}  
\newfont{\scyr}{wncyr10 scaled 550}
\nc{\ssha}{\mbox{\bf \scyr X}}
\nc{\shap}{{\mbox{\cyrs X}}} 
\nc{\shpr}{\diamond}    
\nc{\shp}{\ast} \nc{\shplus}{\shpr^+}
\nc{\shprc}{\shpr_c}    
\nc{\dep}{\mrm{dep}} \nc{\lc}{\lfloor} \nc{\rc}{\rfloor}
\nc{\db}{\leq_{\rm db}} \nc{\bfk}{{\bf k}}

\nc{\cala}{{\mathcal A}} \nc{\calb}{{\mathcal B}}
\nc{\calc}{{\mathcal C}}
\nc{\cald}{{\mathcal D}} \nc{\cale}{{\mathcal E}}
\nc{\calf}{{\mathcal F}} \nc{\calg}{{\mathcal G}}
\nc{\calh}{{\mathcal H}} \nc{\cali}{{\mathcal I}}
\nc{\call}{{\mathcal L}} \nc{\calm}{{\mathcal M}}
\nc{\caln}{{\mathcal N}} \nc{\calo}{{\mathcal O}}
\nc{\calp}{{\mathcal P}} \nc{\calr}{{\mathcal R}}
\nc{\cals}{{\mathcal S}} \nc{\calt}{{\mathcal T}}
\nc{\calu}{{\mathcal U}} \nc{\calw}{{\mathcal W}} \nc{\calk}{{\mathcal K}}
\nc{\calx}{{\mathcal X}} \nc{\CA}{\mathcal{A}}

\nc{\fraka}{{\mathfrak a}} \nc{\frakA}{{\mathfrak A}}
\nc{\frakb}{{\mathfrak b}} \nc{\frakB}{{\mathfrak B}}
\nc{\frakc}{{\mathfrak c}}
\nc{\frakD}{{\mathfrak D}} \nc{\frakF}{\mathfrak{F}}
\nc{\frakf}{{\mathfrak f}} \nc{\frakg}{{\mathfrak g}}
\nc{\frakH}{{\mathfrak H}} \nc{\frakL}{{\mathfrak L}}
\nc{\frakM}{{\mathfrak M}} \nc{\bfrakM}{\overline{\frakM}}
\nc{\frakm}{{\mathfrak m}} \nc{\frakP}{{\mathfrak P}}
\nc{\frakN}{{\mathfrak N}} \nc{\frakp}{{\mathfrak p}}
\nc{\frakS}{{\mathfrak S}} \nc{\frakT}{\mathfrak{T}}
\nc{\frakX}{{\mathfrak X}}

\font\cyr=wncyr10 \font\cyrs=wncyr7
\nc{\li}[1]{\textcolor{blue}{Nan:#1}}
\nc{\lir}[1]{\textcolor{red}{Li:#1}}
\nc{\yi}[1]{\textcolor{blue}{Yi: #1}}
\nc{\xing}[1]{\textcolor{purple}{Xing:#1}}
\nc{\revise}[1]{\textcolor{red}{#1}}
\nc{\nan}[1]{\textcolor{blue}{Nan:#1}}

\numberwithin{equation}{section}
\nc{\RR}{\mathbb{R}}
\nc{\X}{{\bf X}}
\nc{\E}{{\bf E}}
\nc{\x}{\mathbb{X}}
\nc{\C}{\mathcal{C}^{\alpha}}
\nc{\D}{\mathcal{D}^{\alpha}}
\nc{\CC}{\mathcal{C}_{\X}^{\alpha}}
\nc{\f}{\varphi}
\nc{\al}{\alpha}
\nc{\lbar}{\overline}
\nc{\HA}{\mathbb{S}}
\nc{\ha}{\mathcal{S}}
\nc{\cL}{\mathcal L}
\nc{\cD}{\mathcal D}

\nc{\V}{V} \nc{\pro}{\otimes}
\nc{\tng}{T^{\le N}(V)^{g}} \nc{\tn}{T^{\le N}(V)}
\nc{\ttg}{T^{\le 3}(V)^{g}}
\nc{\ZZ}{\mathbb{Z}} \nc{\etree}{1}
\nc{\xx}{\mathcal{X}}
\nc{\RP}{{\mathcal{D}}^{\alpha}([0, T]^2, V)}
\nc{\Y}{{\bf Y}}\nc{\id}{\text{id}} \nc{\Id}{\text{Id}}\nc{\Z}{{\bf Z}}
\nc{\sym}{\operatorname{Sym}} \nc{\tri}{\operatorname{Sym}}
\nc{\norm}[1]{\left\lVert#1\right\rVert}

\begin{document}

\title[Path-dependent RDEs in the full step-two range]{Path-dependent rough differential equations in the full step-two range}

\author{Xing Gao } 
\address{School of Mathematics and Statistics, Lanzhou University
Lanzhou, 730000, China; Gansu Provincial Research Center for Basic Disciplines of Mathematics
and Statistics, Lanzhou, 730070, China
}
\email{gaoxing@lzu.edu.cn}

\author{Nannan Li}
\address{School of Mathematics and Statistics, Lanzhou University
Lanzhou, 730000, China
}
\email{linn2024@lzu.edu.cn}

\begin{abstract}
We establish the existence of solutions to path-dependent rough differential equations with non-anticipative coefficients whose regularity is formulated in terms of horizontal and vertical Dupire derivatives. 
Under the same regularity assumptions as in Ananova~\cite{An23}, we develop a new causal Taylor argument and establish a second-order remainder estimate.
This removes the regularity loss responsible for the restriction $2\leq p<1+\sqrt{2}$ in~\cite{An23} and extends the existence result to the full range $2\le p<3$.
\end{abstract}

\makeatletter
\@namedef{subjclassname@2020}{\textup{2020} Mathematics Subject Classification}
\makeatother
\subjclass[2020]{
60L20, 
60H05, 
60H10, 
60H99, 
60L90. 
}

\keywords{rough path, path-dependent  coefficients, rough integral, rough differential equations.}

\maketitle

\tableofcontents

\setcounter{section}{0}

\allowdisplaybreaks

\section{Introduction}
Rough path theory, initiated by Lyons~\cite{Ly98}, provides a pathwise framework for differential equations of the form
\[
dY(t)=F\big(Y(t)\big)\,dX(t),\quad Y(0)=y\in \RR^m
\]
driven by highly irregular signals $X$, for which classical Riemann--Stieltjes or Young integration is no longer applicable.
Building on the foundational work of Lyons~\cite{Ly98}, rough path theory has developed into a general pathwise framework for the analysis of differential equations driven by irregular signals; see~ \cite{FH20, FV10}.
An important development was Gubinelli's theory of controlled rough paths~\cite{Gu04}, which provided a flexible analytic formulation of rough integration and rough differential equations. 
The theory has since been extended in several directions, including stochastic differential equations driven by a wide class of irregular processes, non-geometric and branched rough paths, discontinuous rough paths, and rough partial differential equations~\cite{FZ18, Gu10,Ha11, Ha13, HK15,LGM26}.
These developments have made rough path theory a robust tool in stochastic analysis and in the study of differential equations under highly irregular forcing.

Path dependence arises naturally in stochastic differential equations when the evolution of the system depends not only on its current state but also on its past trajectory.
Such equations are usually referred to as stochastic functional differential equations or stochastic delay differential equations, and have a long history in stochastic analysis and stochastic modelling; see~\cite{IN64, Ku68, Mo98}. 
The connection between stochastic differential equations and rough differential equations therefore naturally leads to rough equations with memory and path-dependent coefficients.

Several approaches to such equations have been developed. One classical strategy is to lift a path-dependent equation to an infinite-dimensional equation on a space of paths~\cite{Mo98}.
This approach was extended to rough differential equations in Banach spaces by Bailleul \cite{Ba14}, but it relies on Fr\'echet differentiability of the vector fields, thereby excluding many examples.
More specific forms of path dependence have also been studied extensively.
Neuenkirch, Nourdin and Tindel~\cite{NNT08} established existence, uniqueness and stability results for delay rough differential equations, while path-dependent bounded variation terms were considered by Aida and Deya et al.~\cite{Ai24, DGH19}. 
Keller and Zhang~\cite{KZ16} studied a different form of path dependence, where the coefficients depend on the driving rough path rather than on the solution. 
More recently, Kwossek, Neuenkirch and Pr\"omel~\cite{KNP25} established existence, uniqueness and stability results for rough functional differential equations with path-dependent coefficients driven by c\`adl\`ag rough paths under assumptions formulated directly on controlled-path spaces. 

A different route was taken by Ananova~\cite{An23}, who studied path-dependent RDEs through the non-anticipative functional calculus of Cont--Fourni\'e and Dupire~\cite{CF10b,CF13,Du19}.
In this framework, regularity is expressed in terms of horizontal and vertical derivatives on the stopped-path space, rather than Fr\'echet derivatives with respect to arbitrary perturbations of the whole history.
Based on directional derivatives rather than Fr\'echet differentiability, this functional calculus applies to a broader class of path-dependent ODEs and SDEs~\cite{CK20}.

The present work is primarily concerned with sharpening Ananova's existence result within this framework. 
We consider the following integral formulation of the path-dependent rough differential equation~\eqref{eq:5.1}:
\begin{equation}\label{eq:01}
Y(t) = \xi_{t_0}+\int_{t_0}^t b(s,Y_s)\,ds+\int_{t_0}^t \sigma(s,Y_s)\,d\mathbf X(s),
\quad t\in[t_0,T],
\end{equation}
where $\mathbf X=(X,\mathbb X)$ is a $p$-rough path with $2\le p<3$, and $b,\sigma$ are non-anticipative functionals.
We work under the same Dupire regularity assumptions as in~\cite{An23}.

Under these assumptions, Ananova~\cite{An23} proved existence of solutions to \eqref{eq:01} for $2\le p<1+\sqrt2$.
The restriction does not arise from rough integration itself, but from the regularity obtained for the functional composition appearing in the controlled rough path formulation. 
More precisely, for a regular non-anticipative functional $F$, the remainder
\[
F(t,x_t)-F(s,x_s)-\nabla_x F(s,x_s)x_{s,t}
\]
is shown in~\cite{An23} to have finite $q_p$-variation, where $q_p=p^2/(p+1)$ and $\nabla_x F$ denotes the Dupire vertical derivative of $F$.
The corresponding sewing condition is
\[
\frac1p+\frac1{q_p}>1,
\]
which is exactly equivalent to $p<1+\sqrt2$.
Thus the upper bound $1+\sqrt2$ reflects a loss in the regularity of the functional remainder rather than an intrinsic restriction of the rough differential equation.

The main purpose of the present paper is to remove this loss without strengthening the assumptions on the coefficients. 
We prove, under the same Dupire regularity conditions as in~\cite{An23}, that the functional remainder in fact has the canonical $p/2$-variation regularity for every $2\le p<3$. 
Our proof employs a new causal Taylor argument that differs from the approach in~\cite{An23}.

This improvement is the central analytical point of the paper.
Within the controlled rough path framework, removing this loss allows us to reach the full range $2\le p<3$. 
Once the canonical $p/2$-variation estimate is available, regular non-anticipative functionals preserve the standard controlled rough path structure, and the condition required by the sewing lemma becomes
\[
\frac1p+\frac2p>1,
\]
which is exactly the condition $p<3$.

We then combine the composition theorem for controlled rough paths with a Schauder fixed point argument to establish the existence of solutions to \eqref{eq:01}.
Consequently, the existence result of~\cite{An23} extends from $2\le p<1+\sqrt2$ to the full range $2\le p<3$ under the same Dupire regularity assumptions.

The paper is organized as follows.
Section~\ref{sec:2} recalls the non-anticipative functionals framework, Dupire derivatives and rough paths.
Section~\ref{sec:3} proves the new causal Taylor estimate.
Section~\ref{sec:4} establishes the corresponding composition and continuity results for controlled rough paths.
Section~\ref{sec:5} applies these results to establish the existence of solutions to \eqref{eq:01} via the Schauder fixed point theorem.

\section{Preliminaries}\label{sec:2}
We collect in this section the basic notions of rough path theory and non-anticipative functionals that will be used throughout the paper. 

\subsection{Rough paths}\label{sec:2.1}
Let $E$ be a finite-dimensional normed vector space. 
For a continuous path $X:[0,T]\to E$, we use the increment notation
\[
X_{s,t}:=X_t-X_s,\quad 0\le s\le t\le T.
\]
Its $p$-variation seminorm is defined by
\[
\|X\|_{p\text{-var};[0, T]}:=\Big(\sup_{\Pi} \sum_{[u,v]\in\Pi}|X_{u,v}|^{p}\Big)^{1/p},
\]
where the supremum is taken over all finite partitions $ \Pi=\{0=t_0<t_1<\cdots<t_n=T\}$ of $[0, T]$. 
We denote by $C^{p\text{-var}}([0,T],E)$ the space of continuous paths with finite $p$-variation. For a two parameter map
\[
R:\Delta_T\to E,\quad \Delta_T:=\{(s,t)\mid 0\le s\le t\le T\},
\]
we similarly set
\[
\|R\|_{p\text{-var};[0, T]}:=\Big(\sup_{\Pi}\sum_{[u,v]\in\Pi}|R_{u,v}|^{p}\Big)^{1/p}.
\]

It will be convenient to associate with $X$ the interval function
\[
\omega_X(s,t):=\|X\|_{p\text{-var};[s,t]}^{p}.
\]
By definition, we can obtain
\begin{equation*}
|X_{s,t}|^{p}\le\omega_X(s,t),
\end{equation*}
and
\begin{equation*}
\omega_X(s,u)+\omega_X(u,t)\le\omega_X(s,t),\quad 0\le s\le u\le t\le T.
\end{equation*}
This naturally leads to the notion of a control.

\begin{definition}\label{def:1}  
A map $\omega:\Delta_T\to[0,\infty)$ is called a \emph{control} if $\omega(t,t)=0$ for every $t\in[0,T]$ and
\[
\omega(s,u)+\omega(u,t)\le\omega(s,t)
\]
whenever $0\le s\le u\le t\le T$.
\end{definition}

We shall repeatedly use the elementary fact that if $\omega_1$ and $\omega_2$ are controls, then so are
$\omega_1^{\,r}$ with $r\ge1$ 
and $\omega_1^{\,\theta}\omega_2^{\,1-\theta}$ with $\theta\in(0,1)$.
The notion of a control gives an equivalent characterization of finite $p$-variation.

\begin{lemma}\label{lem:5}   
A continuous path $X:[0,T]\to E$ belongs to $C^{p\text{-var}}([0,T],E)$ if and only if there exists a control $\omega$ and constant $C$ such that
\[
|X_{s,t}|\le C\omega(s,t)^{\frac{1}{p}}, \quad 0\le s\le t\le T.
\]
Furthermore, if $|X_{s,t}|\le C\omega(s,t)^{1/p}$, then
\[
\norm{X}_{p\text{-var};[s,t]}\le C\omega(s,t)^{1/p}.
\]
\end{lemma}

\begin{proof}
If $X$ has finite $p$-variation, take $\omega(s,t)=\|X\|_{p\text{-var};[s,t]}^p$.
Conversely, for any partition $\Pi$ of $[s,t]$,
\[
\sum_{[u,v]\in\Pi}|X_{u,v}|^p\le C^p\sum_{[u,v]\in\Pi}\omega(u,v)\le C^p\omega(s,t).
\]
Taking the supremum over $\Pi$ completes the proof.
\end{proof}

Thus, estimates in $p$-variation may be viewed as the analogue of H\"older estimates with the deterministic quantity $|t-s|$ replaced by a general control. The following conclusion provides the quantitative estimates needed for the subsequent balancing argument.

\begin{proposition}\label{prop:1}
Let $p\ge2$, $X\in C^{p\text{-var}}([0,T], E)$, $[s,t]\subset[0,T]$, and $N\ge1$. There is a continuous path $X^{N;s,t}$ which agrees with $X$ on $[0,s]$, is polygonal with at most $N$ affine pieces on $[s,t]$, is extended constantly on $[t,T]$, and satisfies
$ X^{N;s,t}(s)=X(s)$, $X^{N;s,t}(t)=X(t)$. Then

\begin{align}
\norm{X^{N;s,t}-X}_{\infty;[0,t]}&\le 2N^{-1/p}\norm{X}_{p\text{-var};[s,t]},\label{eq:2.1}\\
\norm{X^{N;s,t}}_{1\text{-var};[s,t]}&\le N^{1-1/p}\norm{X}_{p\text{-var};[s,t]},\label{eq:2.2}\\
\norm{X^{N;s,t}}_{p\text{-var};[s,t]}&\le C_p\norm{X}_{p\text{-var};[s,t]},\label{eq:2.3}\\
\norm{X^{N;s,t}}_{p/2\text{-var};[s,t]}&\le C_pN^{1/p}\norm{X}_{p\text{-var};[s,t]},\label{eq:2.4}
\end{align}
where $C_p\in \RR$.
\end{proposition}
\begin{proof}
Let $$\omega_X(u,v)=\norm{X}_{p\text{-var};[u,v]}^p.$$ Choose points $s=u_0\le u_1\le\cdots\le u_N=t$ such that
$$\omega_X(u_i,u_{i+1})\le \omega_X(s,t)/N;$$ repeated points may be removed. Interpolate linearly between $X(u_i)$ and $X(u_{i+1})$. On each interval,
\[
\sup_{r\in[u_i,u_{i+1}]}|X(r)-X^{N;s,t}(r)|\le2\omega_X(u_i,u_{i+1})^{1/p},
\]
which gives \eqref{eq:2.1}. Moreover,
\[
\norm{X^{N;s,t}}_{1\text{-var};[s,t]}=\sum_i|X_{u_i,u_{i+1}}|\le N^{1-1/p}\Big(\sum_i|X_{u_i,u_{i+1}}|^p\Big)^{1/p},
\]
and the last factor is at most $\norm{X}_{p\text{-var};[s,t]}$. This proves \eqref{eq:2.2}.

We next prove the $p$-variation stability. 
Let $Z=X^{N;s,t}$ and take an arbitrary partition $\pi=(t_j)$. 
Intervals lying in a single piece are treated directly by linearity, since their total contribution is bounded by the corresponding mesh increment.
We choose $0\le a<b+1\le N$ such that $t_j\in [u_a,u_{a+1}]$ and $t_{j+1}\in [u_b,u_{b+1}]$. 
Then
\[
Z_{t_j, t_{j+1}}=Z_{t_j, u_{a+1}}+Z_{u_{a+1}, u_b}+Z_{u_b, t_{j+1}}.
\]
By the definition of $X^{N;s,t}$, there exist $\alpha_j, \beta_j\in [0, 1]$ such that
\[
Z_{t_j, t_{j+1}}=\alpha_jZ_{u_a, u_{a+1}}+Z_{u_{a+1}, u_b}+\beta_j Z_{u_b, u_{b+1}}.
\]
Using the inequality $$|a+b+c|^p\le 3^{p-1}(|a|^p+|b|^p+|c|^p),$$ we have
\begin{align*}
|Z_{t_j, t_{j+1}}|^p
&\ \le 3^{p-1}\Big(\big|\alpha_jZ_{u_a, u_{a+1}}\big|^p+|Z_{u_{a+1}, u_b}|^p+ \big|\beta_j Z_{u_b, u_{b+1}}\big|^p  \Big)\\
&\ \le 3^{p-1}\Big(|Z_{u_a, u_{a+1}}|^p+|Z_{u_{a+1}, u_b}|^p+ |Z_{u_b, u_{b+1}}|^p  \Big) .
\end{align*}
Hence
\[
\sum_j|Z_{t_j,t_{j+1}}|^p\le C_p\bigg(\norm{X}_{p\text{-var};[s,t]}^p+\sum_i|X_{u_i,u_{i+1}}|^p\bigg)
\le C_p\norm{X}_{p\text{-var};[s,t]}^p.
\]
Taking the supremum proves \eqref{eq:2.3}.

Finally, variation interpolation between exponents $1$ and $p$ gives
\[
\norm{Z}_{p/2\text{-var}}\le \norm{Z}_{1\text{-var}}^{1/(p-1)}\norm{Z}_{p\text{-var}}^{(p-2)/(p-1)}.
\]
Combining \eqref{eq:2.2} and \eqref{eq:2.3}, and observing that $(1-1/p)/(p-1)=1/p$, proves \eqref{eq:2.4}.  
\end{proof}

We now recall the definition of a $p$-rough path for $2\le p<3$.

\begin{definition}~\cite[Definition 2.1]{FH20}
\label{def:2}
Let $p\in[2,3)$. A \emph{$p$-rough path} over $\mathbb{R}^{d}$ is $\mathbf{X}=(X,\mathbb{X})$, where
\[
X:[0,T]\to\mathbb{R}^{d}, \quad\mathbb{X}:\Delta_T\to \mathbb{R}^{d}\otimes\mathbb{R}^{d}
\]
such that
\begin{enumerate}
\item Chen's relation holds:
\begin{equation*}
\mathbb{X}_{s,t}=\mathbb{X}_{s,u}+\mathbb{X}_{u,t}+X_{s,u}\otimes X_{u,t}, \quad s\le u\le t;
\end{equation*}
\item the first and second levels have finite variation of the corresponding orders:
\[
\|X\|_{p\text{-var};[0,T]} + \|\mathbb{X}\|_{\frac p2\text{-var};[0,T]}<\infty.
\]
\end{enumerate}
The subclass of continuous $p$-rough paths is denoted by $\mathcal{C}^{p\text{-var}}([0,T],\mathbb{R}^{d})$.
\end{definition}

The second level $\mathbb{X}$ contains precisely the additional information required to integrate against $X$ when the regularity of $X$ is too low for Young integration. 
In general, the lift of a path $X$ to a rough path $\mathbf X=(X,\mathbb X)$ is not unique; consequently, the second level is part of the driving signal rather than an object determined by $X$ alone.

The appropriate class of integrands for rough integration is the space of controlled paths introduced by Gubinelli\cite[Definition 1]{Gu04}.  Here we recall Ananova's analogue of controlled rough paths~\cite[Definition 2.4]{An23}.

\begin{definition}
\label{def:3}
Let $p>q\ge1$ and let $X\in C^{p\text{-var}}([0,T],\mathbb{R}^{d})$. A pair $(Y,Y')$ with
\[
Y\in C^{p\text{-var}}([0,T],\mathbb{R}^{k}),\quad Y'\in C^{p\text{-var}} \bigl([0,T],\mathcal{L}(\mathbb{R}^{d},\mathbb{R}^{k})\bigr)
\]
is said to be a \emph{$(p,q)$-controlled path} over $X$ if the remainder
\begin{equation}\label{eq:2.7}
R^{Y}_{s,t}:=Y_{s,t}-Y'_sX_{s,t}
\end{equation}
has finite $q$-variation. The corresponding space is denoted by $ \mathcal{D}^{p,q}_{X}([0,T],\mathbb{R}^{k})$. The path $Y'$ is called the \emph{Gubinelli derivative} of $Y$ with respect to $X$.
\end{definition}

The decomposition
\[
Y_{s,t} =Y'_sX_{s,t}+R^Y_{s,t}
\]
may be interpreted as a first-order expansion of $Y$ along the reference path $X$. 
The essential requirement is that the remainder is more regular than the leading term, which is encoded by $q<p$. 
The space $\mathcal D_X^{p,q}([0,T],\mathbb{R}^{k})$ becomes a Banach space when equipped with the norm
\begin{equation*}
\|(Y,Y')\|_{\mathcal D_X^{p,q}}:=|Y_0|+|Y'_0|+\|Y'\|_{p\text{-var};[0,T]}+\|R^Y\|_{q\text{-var};[0,T]}.
\end{equation*}
For later use, we also introduce the seminorm
\begin{equation*}
\|(Y,Y')\|_{p,q,X}:=\|Y'\|_{p\text{-var};[0,T]}+\|R^Y\|_{q\text{-var};[0,T]}.
\end{equation*}

We now recall the basic integration result for controlled paths. 
It is the key analytic input used in the sequel.

\begin{lemma}~\cite[Theorem 2.5]{An23}\label{lem:1}
Let $p\in[2,3)$ and let $\mathbf X=(X,\mathbb X) \in\mathcal C^{p\text{-var}}([0,T],\mathbb R^d)$. Suppose that $q\ge p/2$ and
$p^{-1}+q^{-1}>1$. Let $(Y,Y')\in\mathcal D_X^{p,q} \bigl([0,T], \mathcal L(\mathbb R^d,\mathbb R^k)\bigr)$. For a partition $\Pi$ of $[0,T]$, consider the compensated sums
\begin{equation*}
S(\Pi):=\sum_{[u,v]\in\Pi}\left(Y_uX_{u,v} +Y'_u\mathbb X_{u,v}\right).
\end{equation*}
Then 
\begin{enumerate}
\item the limit
\[
\int_0^T Y_r\,d\mathbf X_r:=\lim_{|\Pi|\to0}S(\Pi)
\]
exists. 

\item Moreover, for every $0\le s\le t\le T$,
\begin{equation*}
\bigg| \int_s^t Y_r\,d\mathbf X_r-Y_sX_{s,t}-Y'_s\mathbb X_{s,t}\bigg|\le C_{p,q,d}\bigg( \|R^Y\|_{q\text{-var};[s,t]} \|X\|_{p\text{-var};[s,t]}+\|Y'\|_{p\text{-var};[s,t]}\|\mathbb X\|_{\frac p2\text{-var};[s,t]}\bigg).
\end{equation*}

\item If
\[
Z_t:=\int_0^tY_r\,d\mathbf X_r,\quad  Z'_t:=Y_t,
\]
then $$(Z,Z')\in\mathcal{D}^{p,q}_{X}([0,T],\mathbb{R}^{k})$$ and
\begin{align*}
\|(Z,Z')\|_{p,q,X}
\le&\  \|Y\|_{p\text{-var};[0,T]}+ \|Y'\|_{\infty;[0,T]} \|\mathbb X\|_{\frac p2\text{-var};[0,T]}\\
&\ + C_{p,q,d}\left(\|X\|_{p\text{-var};[0,T]} \|R^Y\|_{q\text{-var};[0,T]}+\|Y'\|_{p\text{-var};[0,T]}\|\mathbb X\|_{\frac p2\text{-var};[0,T]}\right).
\end{align*}

\item In particular, the rough integration map
\[
(Y,Y')\mapsto\left(\int_0^\cdot Y_r\,d\mathbf X_r,Y\right)
\]
is continuous on the corresponding space of controlled paths.
\end{enumerate}
\end{lemma}

\subsection{Non-anticipative functionals}\label{sec:2.2}
In this subsection, we recall the notions of non-anticipative functionals and their regularity following~\cite{An23,CF10a}.

Let $D([0,T],\mathbb{R}^{d})$ denote the space of $\mathbb{R}^{d}$-valued c\`adl\`ag paths on $[0,T]$. 
For $x\in D([0,T],\mathbb{R}^{d})$ and $t\in[0,T]$, we write
\[
x_t(s):=x(s\wedge t), \quad s\in[0,T],
\]
for the path stopped at time $t$. 
A functional
\[
    F:[0,T]\times D([0,T],\mathbb{R}^{d})\to\mathbb{R}^{k}
\]
is called \emph{non-anticipative} if its value at time $t$ depends only on the history of the path up to $t$, namely,
\begin{equation}
F(t,x)=F(t,x_t), \quad (t,x)\in[0,T]\times D([0,T],\mathbb{R}^{d}).
\label{eq:2.5}
\end{equation}
Accordingly, it is natural to regard $F$ as a function defined on the space of stopped paths. 
More precisely, on $[0,T]\times D([0,T],\mathbb{R}^{d})$ we introduce the equivalence relation
\[
(t,x)\sim(t',x')
\]
whenever $t=t'$ and $x_t=x'_{t'}$. 
The corresponding quotient space is denoted by $\Lambda_T^d$. 
It is endowed with the metric
\begin{equation}
\begin{aligned}
d_\infty\big((t,x),(t',x')\big):=&\sup_{u\in[0,T]}\big|x(u\wedge t)-x'(u\wedge t')\big|+|t-t'|\\
=&\  \|x_t-x'_{t'}\|_\infty+|t-t'|.
\end{aligned}
\label{eq:2.6}
\end{equation}
With this distance, $(\Lambda_T^d,d_\infty)$ is a complete metric space. 
Hence every functional satisfying \eqref{eq:2.5} may be identified with a functional on $\Lambda_T^d$.

\begin{definition}\label{def:4} \cite{An23, CF10a}
A measurable mapping $F:\Lambda_T^d\to\mathbb{R}^{k}$ is called a \emph{non-anticipative functional}. 
We denote by $C^{0,0}(\Lambda_T^d)$ the class of functionals that are continuous with respect to $d_\infty$.
\end{definition}

Continuity in the metric $d_\infty$ implies joint continuity with respect to time and the stopped path. 
In what follows, we shall also use weaker continuity properties.

\begin{definition}\label{def:5} \cite{An23, CF10a}
Let $F:\Lambda_T^d\to\mathbb{R}^{k}$ be non-anticipative. We say that:

\begin{enumerate}
\item $F$ is \emph{continuous at fixed times} if, for every $t\in[0,T]$, the mapping
\[
F(t, \cdot) :D([0,T],\mathbb{R}^{d})\to \mathbb{R}^{k},\quad  x\mapsto F(t,x)
\]
is continuous with respect to the norm $\|\cdot\|_{\infty}$. Equivalently, for every $x\in D([0,T],\mathbb{R}^{d})$ and $\varepsilon>0$, there exists $\delta>0$ such that
\[
\|x_t-x'_t\|_\infty<\delta \quad\Longrightarrow\quad |F(t,x)-F(t,x')|<\varepsilon.
\]

\item $F$ is \emph{left-continuous} if, for every $(t,x)\in\Lambda_T^d$ and $\varepsilon>0$, there exists $\delta>0$ such that
\[
t'<t, \quad d_\infty\big((t,x),(t',x')\big)<\delta
\]
imply
\[
|F(t,x)-F(t',x')|<\varepsilon.
\]
The class of left-continuous non-anticipative functionals is denoted by $C_l^{0,0}(\Lambda_T^d)$.
\end{enumerate}
\end{definition}

We shall also require a local boundedness property adapted to the space of stopped paths.

\begin{definition}\label{def:6} \cite{An23, CF10a}
A non-anticipative functional $F:\Lambda_T^d\to\mathbb{R}^{k}$ is said to be \emph{boundedness-preserving} if, for every compact set $K\subset\mathbb{R}^{d}$ and every $t_0\in[0,T]$, there exists a constant $C_{K,t_0}>0$ such that
\[
|F(t,x)|\le C_{K,t_0}
\]
whenever $t\in[0,t_0]$ and $x([0,t])\subset K$.
\end{definition}

We now recall the horizontal and vertical derivatives of a non-anticipative functional~\cite{Du19}.
These derivatives describe two distinct perturbations of a stopped path: an extension of the time horizon while keeping the path frozen, and a perturbation of the terminal value.
Given $(t,x)\in\Lambda_T^d$ and $e\in\mathbb{R}^{d}$, define the vertically perturbed path by
\begin{equation*}
x_t^e:=x_t+e\mathbbm{1}_{[t,T]}.
\end{equation*}

\begin{definition}\label{def:7} \cite{Du19}
Let $F:\Lambda_T^d\to\mathbb{R}$ be a non-anticipative functional.
\begin{enumerate}
\item We say that $F$ is \emph{horizontally differentiable} at $(t,x)$ if the right derivative
\begin{equation*}
DF(t,x):=\lim_{h\downarrow0}\frac{F(t+h,x_t)-F(t,x)}{h}
\end{equation*}
exists. 
If this limit exists for every $(t,x)\in\Lambda_T^d$, then $DF$ defines another non-anticipative functional and is called the
\emph{horizontal derivative} of $F$.

\item We say that $F$ is \emph{vertically differentiable} at $(t,x)$ if the mapping
\[
g_{(t, x)}:\RR^d\to \RR, \quad e\mapsto F(t,x_t^e),
\]
is differentiable at $e=0$. Its gradient at the origin is denoted by
\begin{equation*}
\nabla_x F(t,x):=\nabla g_{(t, x)}(0)\in\mathbb{R}^{d},
\end{equation*}
and is called the \emph{vertical derivative} (or \emph{Dupire derivative}) of $F$.
\end{enumerate}
\end{definition}

Writing
\[
\nabla_xF(t,x)=\big(\partial_1F(t,x),\ldots,\partial_dF(t,x)\big),
\]
the $i$-th vertical derivative is explicitly given by
\begin{equation*}
\partial_iF(t,x)=\lim_{h\to0}\frac{F\bigl(t,x_t+h e_i\mathbf{1}_{[t,T]}\bigr)-F(t,x_t)}{h},
\end{equation*}
where $(e_1,\ldots,e_d)$ denotes the canonical basis of $\mathbb{R}^{d}$. 
Throughout, we identify the Euclidean gradient with the corresponding linear functional via the canonical inner product. 
In particular, for $v\in\mathbb{R}^d$, we write
\[
\nabla_x F(t,x)[v]:=\langle \nabla_x F(t,x),v\rangle.
\]
Likewise, we identify the Hessian with the associated bilinear form, and for $u,v\in\mathbb{R}^d$ write
\[
\nabla_x^2 F(t,x)[u,v]:=\left\langle \nabla_x^2 F(t,x)u,v\right\rangle.
\]

Since $DF$ and $\nabla_xF$, whenever they exist, are themselves non-anticipative functionals, the differentiation procedure can be iterated. 
In particular, higher-order vertical derivatives such as $\nabla_x^2F$ are defined in the natural way. This leads to the following class of regular functionals.

\begin{definition} \label{def:8} \cite{An23, CF10a}
For $k\ge1$, we denote by $C_b^{1,k}(\Lambda_T^d)$ the collection of non-anticipative functionals $F:\Lambda_T^d\to\mathbb{R}$ satisfying the following properties:
\begin{enumerate}
\item $F$ is horizontally differentiable and $DF$ is continuous at fixed times;

\item $F$ is $k$ times vertically differentiable and
\[
\nabla_x^jF\in C_l^{0,0}(\Lambda_T^d),\quad j=0,\ldots,k;
\]

\item the functionals $DF,\, \nabla_xF,\,\ldots,\,\nabla_x^kF$ preserve boundedness.
\end{enumerate}
\end{definition}

We denote by $C_b^{0,1}(\Lambda_T^d)$ the class of $F\in C^{0,0}(\Lambda_T^d)$ with left-continuous vertical derivatives, both $F$ and $\nabla_x F$ being boundedness-preserving.
Following Ananova \cite{An23}, we impose Lipschitz continuity with respect to the stopped-path metric.

\begin{assumption} \label{assum:1}
There exists a constant $K>0$ such that, for every $(t,x),(t',x')\in\Lambda_T^d$,
\begin{equation*}
|F(t,x)-F(t',x')|\le K\,d_\infty\big((t,x),(t',x')\big).
\end{equation*}
We denote the class of such functionals by $\operatorname{Lip}(\Lambda_T^d,d_\infty)$.
\end{assumption}

For later estimates, it is useful to separate the dependence on the path variable from the dependence on time.

\begin{assumption}\label{assum:2}
There exists $K>0$ such that, for every $x,x'\in D([0,T],\mathbb{R}^{d})$ and $t\in[0,T]$,
\begin{equation*}
|F(t,x)-F(t,x')|\le K\|x_t-x'_t\|_\infty.
\end{equation*}
The corresponding class is denoted by $\operatorname{Lip}(\Lambda_T^d,\|\cdot\|_\infty)$.
\end{assumption}

\begin{assumption}\label{assum:3}
There exists $K>0$ such that, for every $x\in D([0,T],\mathbb{R}^{d})$ and $0\le t'\le t\le T$,
\begin{equation*}
|F(t,x_{t'})-F(t',x_{t'})|\le K|t-t'|.
\end{equation*}
We write $\operatorname{hLip}(\Lambda_T^d)$ for the corresponding class of functionals.
\end{assumption}

By the definition of $d_\infty$, Lipschitz continuity with respect to the stopped-path metric is equivalent to requiring both uniform Lipschitz continuity in the path variable and Lipschitz continuity in the time variable. 
Consequently,
\begin{equation*}
\operatorname{Lip}(\Lambda_T^d,d_\infty)=\operatorname{Lip}(\Lambda_T^d,\|\cdot\|_\infty)\cap\operatorname{hLip}(\Lambda_T^d).
\end{equation*}

\section{The canonical causal Taylor estimate}\label{sec:3}
In this section, we establish a new causal Taylor remainder estimate for non-anticipative functionals.

In what follows, we write $\nabla$ in place of $\nabla_x$ for simplicity. 
For $s\in[0,T]$, let $\eta\in C([0,s],\mathbb{R}^k)$ be a continuous history and $z\in C([s,T],\mathbb{R}^k)$ a continuous path satisfying $z(s)=\eta(s)$. 
We denote by $\eta\oplus_s z$ their concatenation at time $s$. 
We first record a localized version of the finite variation chain rule, which requires finite variation only on the future part $z$, while the past history $\eta$ is merely continuous.

\begin{lemma}\label{lem:2}
Assume that $F\in C_b^{1,1}(\Lambda_T^k)$ and that $F,DF,\nabla F$ are continuous in $d_\infty$. Let $x=\eta\oplus_s z$, where $\eta$ is continuous and $z$ is continuous of finite variation on $[s, T]$. Then, for any $t\in (s, T]$,
\begin{equation*}
F(t,x_t)-F(s,x_s)=\int_s^t DF(r,x_r)\,d r+\int_s^t\nabla F(r,x_r)\,d z_r.
\end{equation*}
\end{lemma}

\begin{proof}
Choose polygonal histories $\eta^n$ on $[0,s]$ such that $\eta^n(s)=\eta(s)$ and $\eta^n\to\eta$ uniformly.  Set $x^n=\eta^n\oplus_s z$. Each $x^n$ has finite variation on $[0,t]$, so the standard functional finite variation chain rule (see~\cite{CF10a}) applies on $[s,t]$:
\[
F(t,x^n_t)-F(s,x^n_s)=\int_s^tDF(r,x^n_r)\,d r+\int_s^t\nabla F(r,x^n_r)\,d z_r.
\]
For $r\in[s,t]$, 
\[
d_\infty\big((r,x^n_r),(r,x_r)\big)=\|x^n_t-x_t\|_\infty=\|\eta^n-\eta\|_\infty.
\]  
Hence the integrands converge uniformly by continuity.  The Lebesgue integral converges by dominated convergence and the Stieltjes integral converges because its integrator $z$ is fixed and of finite variation.  The endpoint terms converge as well. 
\end{proof}

We impose the following regularity assumption throughout the paper.

\begin{assumption}\label{assum:4}
We assume that the vector field $F:\Lambda_T^k\to\RR^d$ satisfies
\[
F\in C_b^{1,2}(\Lambda_T^k, \RR^d), \quad \nabla F\in C_b^{1,1}(\Lambda_T^k;\cL(\RR^k, \RR^d)),
\]
and 
\[
F,\quad DF,\quad \nabla F,\quad D\nabla F,\quad \nabla^2F
\]
are globally Lipschitz in $d_\infty$.
\end{assumption}

This assumption is given based on the regularity in the sense of Dupire~\cite{CF10b, Du19}, and is the same as~\cite[Assumption 4.3]{An23}. 
We now record the second order identity from which the estimate follows. Set
\[
\mathcal T^F_{s,t}(x):=F(t,x_t)-F(s,x_s)-\nabla_x F(s,x_s)[x_{s, t}].
\]

\begin{proposition}\label{prop:2}
Let $F$ satisfy  Assumption~\ref{assum:4} and let $x:[0,T]\to\RR^k$ be continuous of finite variation on $[s,t]\subseteq  [0,T]$. Then
\begin{equation*}
\mathcal T^F_{s,t}(x)=\int_s^t DF(r,x_r)\, d r+\int_s^t D\nabla F(r,x_r)[x_{r, t}]\, d r+\int_s^t\nabla^2F(r,x_r)[x_{r, t}, d x_r].
\end{equation*}

\end{proposition}

\begin{proof}
Applying Lemma~\ref{lem:2}  to $F$ on $[s,t]$ yields
\begin{equation}\label{eq:3.1}
F(t,x_t)-F(s,x_s)=\int_s^tDF(r,x_r)\, d r+\int_s^t\nabla F(r,x_r)\, d x_r.
\end{equation}
Since $\nabla F\in C_b^{1,1}(\Lambda_T^k;\cL(\RR^k, \RR^d))$, applying Lemma~\ref{lem:2} componentwise to $\nabla F$ yields
\[
\nabla F(r,x_r)-\nabla F(s,x_s)=\int_s^rD\nabla F(u,x_u)\, d u+\int_s^r\nabla^2F(u,x_u)\, d x_u.
\]
Insert this into the second integral in~\eqref{eq:3.1}. Fubini's theorem for Lebesgue--Stieltjes integrals yields
\[
\int_s^t\int_s^rD\nabla F(u,x_u)\, d u\, d x_r =\int_s^tD\nabla F(u,x_u)[x_{u, t}]\, d u
\]
and
\[
\int_s^t\int_s^r\nabla^2F(u,x_u)\, d x_u\, d x_r=\int_s^t\nabla^2F(u,x_u)[x_{u, t}, d x_u].
\]
Subtracting $\nabla F(s,x_s)[x_{s, t}]$ proves the formula.
\end{proof}

For the subsequent processing, we present some conclusions.

\begin{lemma}\label{lem:4}
For every symmetric bilinear map $H$ and every continuous finite variation path $x$,
\begin{equation*}
\int_s^t H[x_{r, t}, d x_r]=\frac12H[x_{s, t}, x_{s, t}].
\end{equation*}
\end{lemma}

\begin{proof}
Since $r\mapsto x_{r, t}$ has differential $-d x_r$,
\[
dH[x_{r, t}, x_{r, t}] =-2H[x_{r, t}, d x_r].
\]
Integrating from $s$ to $t$ gives the required result.
\end{proof}

For $x\in C^{p\text{-var}}$, define the control
\begin{equation}\label{eq:3.4}
 \omega_x(s,t):=|t-s|+\norm{x}_{p\text{-var};[s,t]}^p.
\end{equation}

\begin{lemma}\label{lem:4+1}
Let $2\le p<3$, let $F$ satisfy Assumption~\ref{assum:4}, and let $x^N=x^{N;s,t}$ be the polygonal approximation from Proposition~\ref{prop:1}. If $\omega:=\omega_x(s,t)\le1$, then
\begin{equation}\label{eq:3.2}
\left|\int_s^t\bigl(\nabla^2F(r,x^N_r)-\nabla^2F(s,x_s)\bigr)[x^N_{r,t},d x^N_r]\right|\le C N^{1/p}\omega^{3/p}.
\end{equation}

\end{lemma}
\begin{proof}
Write $$\widehat H_r:=\nabla^2F(r,x^N_r),\quad z_r:=x^N_{r,t},$$  and define the linear map valued path
\[
 G_r(v):=(\widehat H_r-\widehat H_s)[z_r, v].
\]
Then the integral in \eqref{eq:3.2} is $\int_s^tG_r\,d x^N_r$.

We assert that 
\begin{equation*}
\big\|\widehat H\big\|_{p\text{-var};[s,t]}+\norm{z}_{p\text{-var};[s,t]}\le C\omega^{1/p}.
\end{equation*}
Indeed, by Assumption~\ref{assum:4}, there exist a constant $C>0$ such that 
\[
\big|\widehat H_v-\widehat H_u\big|=\big|\nabla^2F(v,x^N_v)-\nabla^2F(u,x^N_u)\big|\le C\,d_{\infty}\big((v,x^N_v), (u,x^N_u)\big).
\]
By~\eqref{eq:2.6},
\[
\big|\widehat H_v-\widehat H_u\big|\le C\big(|v-u|+\|x^N_v-x^N_u\|_{\infty}\big).
\]
For an arbitrary partition $$s=t_0\le t_1\le\cdots\le t_m=t,$$ we have  
\begin{align*}
\sum_i \big|\widehat H_{t_i, t_{i+1}}\big|^p\le&\  C_p\sum_i\big(|t_{i+1}-t_i|^p+\|x^N_{t_{i+1}}-x^N_{t_{i}}\|^p_{\infty}\big)\\
\le&\ C_p\bigg(\Big(\sum_i |t_{i+1}-t_i|\Big)^p+\sum_i \|x^N\|^p_{p\text{-var}; [t_i, t_{i+1}]}  \bigg)\\
\le&\ C_p\Big(|t-s|^p+\|x^N\|^p_{p\text{-var}; [s, t]}\Big)\\
\le&\ C_p\Big(T^{p-1}|t-s|+C_p\|x\|^p_{p\text{-var}; [s, t]}\Big)\hspace{1cm} (\text{by (\ref{eq:2.3})})\\
\le&\ C_p(T^{p-1}+C_p)\Big(|t-s|+\|x\|^p_{p\text{-var}; [s, t]}\Big)\\
=&\ C_p\omega \hspace{1cm} (\text{by setting $C_p:=C_p(T^{p-1}+C_p)$}).
\end{align*}
So $$\|\widehat H\|_{p\text{-var};[s,t]}\le C_p\omega^{1/p}.$$ On the other hand, 
\[
z_{u, v}=z_v-z_u=x^N_{v,t}-x^N_{u,t}=-x^N_{u, v}.
\]
This implies $\|z\|_{p\text{-var}; [s, t]}=\|x^N\|_{p\text{-var}; [s, t]}$. By (\ref{eq:2.3}) and $\|x\|^p_{p\text{-var}; [s, t]}\le \omega$,
we have 
\[
\norm{z}_{p\text{-var};[s,t]}\le C_p\omega^{1/p}.
\]
The product estimate for $p$-variation therefore yields
\begin{equation}\label{eq:3.3}
\|G\|_{p\text{-var};[s,t]}\le C_p\left(\|\widehat H-\widehat H_s\|_\infty\|z\|_{p\text{-var}}+\|z\|_\infty\|\widehat H\|_{p\text{-var}}\right)
\le C_p\omega^{2/p}.
\end{equation}
Since $$1/p+2/p=3/p>1,$$ Young's inequality with exponents $p$ and $p/2$, together with $G_s=0$, gives
\[
\left|\int_s^tG_r\,d x^N_r\right|\le C_p\norm G_{p\text{-var};[s,t]}\|x^N\|_{p/2\text{-var};[s,t]}.
\]
Now use \eqref{eq:2.4} and \eqref{eq:3.3} to obtain the desired estimate.
\end{proof}

Combining the preceding estimates yields the main conclusion.

\begin{theorem}\label{thm:1}
Let $2\le p<3$ and let $F$ satisfy Assumption~\ref{assum:4}. There exists a constant $C=C(F,p,T)$ such that, for every
$x\in C^{p\text{-var}}([0,T],\mathbb{R}^k)$ and every $0\le u\le v\le T$ satisfying $\omega_x(u,v)\le 1$, one has
\begin{equation}\label{eq:3.13}
|\mathcal T^F_{u,v}(x) |\le C \omega_x(u,v)^{2/p}.
\end{equation}
Consequently, possibly after increasing $C$, for every such $x$,
\begin{equation}\label{eq:3.14}
\norm{\mathcal T^F(x)}_{p/2\text{-var};[s,t]}\le C \omega_x(s,t)^{2/p}
\end{equation}
for all $0\le s\le t\le T$. 
\end{theorem}

\begin{proof}
Fix $[u,v]$ and write $\omega=\omega_x(u,v)\le1$.  
If $\omega=0$, then $u=v$ and $x$ is constant on $[u,v]$, so the claim is immediate. 
Assume henceforth that $\omega>0$, let $N=\lceil\omega^{-1}\rceil$, and put $x^N=x^{N;u,v}$. 
Since $x^N$ and $x$ agree on $[0,u]$ and at time $v$, the stopped path Lipschitz continuity of $F$ gives
\begin{equation}
\begin{aligned}
\big|\mathcal T^F_{u,v}(x)-\mathcal T^F_{u,v}(x^N)\big|
=&\  |F(v,x_v)-F(v,x^N_v)|\\
\le&\ C\|x-x^N\|_{\infty;[0,v]}\\
\le&\ C N^{-1/p}\norm{x}_{p\text{-var};[u,v]} \hspace{1cm} (\text{by (\ref{eq:2.1})})\\
\le&\ C N^{-1/p}\omega^{1/p} \hspace{1cm} (\text{by (\ref{eq:3.4})})\\
\le&\ C\omega^{2/p} \hspace{1cm} (\text{by $N\ge\omega^{-1}$}).
\end{aligned}
\label{eq:3.12}
\end{equation}
Applying Proposition~\ref{prop:2} to $x^N$ and writing $H=\nabla^2F$, we have
\begin{equation}\label{eq:3.11}
\mathcal T^F_{u,v}(x^N)=\int_u^v DF(r,x^N_r)\, d r+\int_u^v D\nabla F(r,x^N_r)[x^N_{r, v}]\, d r+\int_u^v H(r,x^N_r)[x^N_{r, v}, d x^N_r].
\end{equation}
We estimate these three integrals one by one. For the first  term,
\begin{equation}\label{eq:3.5}
\bigg|\int_u^v DF(r,x^N_r)\, d r\bigg|\le C|v-u|\le C\omega\le C \omega^{2/p}.
\end{equation}
The last inequality is due to the fact that $ \omega\le 1$ and $ 2/p\le 1$. For the second  term, since
\begin{equation*}
|x^N_{r, v}|\le  \|x^N\|_{p\text{-var};[r,v]}\le \|x^N\|_{p\text{-var};[u,v]} \overset{(\ref{eq:2.3})}{\le} C  \|x\|_{p\text{-var};[u,v]}
\overset{(\ref{eq:3.4})}{\le} C \omega^{1/p},
\end{equation*}
we have 
\begin{equation}\label{eq:3.6}
\bigg|\int_u^v D\nabla F(r,x^N_r)[x^N_{r, v}]\, d r\bigg|\le C \int_u^v |x^N_{r, v}|\, d r\le C\omega^{1/p}|v-u| \overset{(\ref{eq:3.4})}{\le} C\omega^{1+1/p}\le  C\omega^{2/p}.
\end{equation}
The last inequality is due to the fact that $ \omega\le 1$ and $ 2/p\le 1+1/p$. For the third  term, we  break it down into two parts
\begin{equation}\label{eq:3.7}
\bigg|\int_u^v H(r,x^N_r)[x^N_{r, v}, d x^N_r]\bigg|\le \bigg|\int_u^v H(u,x^N_u)[x^N_{r, v}, d x^N_r] \bigg|+\bigg|\int_u^v\bigl(H(r,x^N_r)-H(u,x_u)\bigr)[x^N_{r,v}, d x^N_r] \bigg|.
\end{equation}
For the first  part, by Lemma~\ref{lem:4}, we obtain
\begin{equation}\label{eq:3.8}
\bigg|\int_u^v H(u,x^N_u)[x^N_{r, v}, d x^N_r] \bigg|=\frac12 |H(u,x_u)[x_{u,v},x_{u,v}]|\le C |x_{u,v}|^2\le C  \|x\|^2_{p\text{-var};[u,v]}\overset{(\ref{eq:3.4})}{\le} C \omega^{2/p}.
\end{equation}
For the second  part, by Lemma~\ref{lem:4+1}, we have
\begin{equation}\label{eq:3.9}
\bigg|\int_u^v\bigl(H(r,x^N_r)-H(u,x_u)\bigr)[x^N_{r,v}, d x^N_r] \bigg|\le CN^{1/p}\omega^{3/p}\le C \omega^{2/p}.
\end{equation}
The last inequality is due to the fact that $ N=\lceil\omega^{-1}\rceil\le \omega^{-1}+1\le 2\omega^{-1}$. Substituting \eqref{eq:3.8} and \eqref{eq:3.9} into \eqref{eq:3.7}, we obtain
\begin{equation}\label{eq:3.10}
\bigg|\int_u^v H(r,x^N_r)[x^N_{r, v}, d x^N_r]\bigg|\le C \omega^{2/p}.
\end{equation}
Substituting \eqref{eq:3.5}, \eqref{eq:3.6} and \eqref{eq:3.10} into \eqref{eq:3.11},
\[
|\mathcal T^F_{u,v}(x^N)|\le C \omega^{2/p}.
\]
Together with \eqref{eq:3.12}, this proves \eqref{eq:3.13}. 

To obtain the $p/2$ variation estimate on $[s,t]$, let $\pi=\{t_0,\ldots,t_n\}$ be an arbitrary partition of $[s,t]$. We divide its subintervals $[t_i,t_{i+1}]$ into two classes according to whether $\omega_x(t_i,t_{i+1})\le1$ or $\omega_x(t_i,t_{i+1})>1$. For the first class, the local estimate above yields
\[
\sum_{\omega_x(t_i,t_{i+1})\le1}\left|\mathcal T^F_{t_i,t_{i+1}}(x)\right|^{p/2}\le C\sum_{\omega_x(t_i,t_{i+1})\le1}\omega_x(t_i,t_{i+1})\le C\omega_x(s,t).
\]
For the second class, the global Lipschitz bound on $F$ and $$|x_{t_i,t_{i+1}}|\le \norm{x}_{p\text{-var};[t_i,t_{i+1}]}$$ imply
\[
|\mathcal T^F_{t_i,t_{i+1}}(x)|\le C(F,p,T)\,\omega_x(t_i,t_{i+1})^{2/p}.
\]
Then
\[
\sum_{\omega_x(t_i,t_{i+1})>1}\left|\mathcal T^F_{t_i,t_{i+1}}(x)\right|^{p/2}\le C \omega_x(s,t).
\]
Combining the two estimates and taking the supremum over all partitions $\pi$ of $[s,t]$ gives \eqref{eq:3.14}.
\end{proof}

\begin{remark}
If $F(t,x)=f(t,x(t))$, \eqref{eq:3.14} is the ordinary second-order Taylor estimate. The theorem shows that the same variation order remains valid for the full non-anticipative class under the stated causal derivatives.
\end{remark}

We next show that the causal Taylor remainder is stable under uniform convergence of paths, with convergence in any weaker $q$-variation topology.

\begin{proposition}\label{prop:3}
Let $F$ satisfy Assumption~\ref{assum:4}. Let $x^n, x\in C^{p\text{-var}}([0,T],\RR^k)$ satisfy 
\[
x^n\to x\quad\text{uniformly},\quad \sup_n\norm{x^n}_{p\text{-var};[0,T]}<\infty.
\]
Then $\mathcal T^F(x^n)\to \mathcal T^F(x)$ uniformly for $0\le s\le t\le T$. Moreover, for every $q>p/2$,
\begin{equation*}
\norm{\mathcal T^F(x^n)-\mathcal T^F(x)}_{q\text{-var};[0,T]}\to 0.
\end{equation*}
\end{proposition}

\begin{proof}
Uniform convergence follows from the Lipschitz continuity of $F$ and $\nabla F$:
\begin{align*}
|\mathcal T^F_{s,t}(x^n)-\mathcal T^F_{s,t}(x)|\le C\big(\norm{x^n-x}_\infty+\norm{x^n-x}_\infty|x^n_{s,t}|+|x^n_{s,t}-x_{s,t}|\big).
\end{align*}
The right-hand side tends to zero uniformly. By Theorem~\ref{thm:1}, the $p/2$-variation seminorms are uniformly bounded. Variation interpolation then gives convergence in every strictly larger variation exponent $q$.
\end{proof}

\section{Controlled composition}\label{sec:4}

In this section, we apply the causal Taylor estimate from the previous section to the controlled rough path setting. 

Throughout this section let $q=p/2$. For a path $Y$ of finite $p$-variation, write
\begin{equation}\label{eq:4.1}
\omega_Y(s,t):=|t-s|+\norm{Y}_{p\text{-var};[s,t]}^p.
\end{equation}
Using the causal Taylor estimate from the previous section, we first show that composition with a non-anticipative functional preserves the controlled rough path structure.

\begin{theorem}\label{thm:2}
Let $F:\Lambda_T^k\to\cL(\RR^d,\RR^m)$ satisfy Assumption~\ref{assum:4}, let $\X$ be a continuous $p$-rough path, and let $(Y,Y')\in\cD_X^{p,q}([0,T],\RR^{k})$.  Define
\begin{equation*}
H_t:=F(t,Y_t),\quad H'_t:=\nabla F(t,Y_t)Y'_t.
\end{equation*}
Then $(H,H')\in\cD_X^{p,q}\bigl([0,T], \mathcal L(\mathbb R^d,\mathbb R^m)\bigr)$.  More precisely,
\begin{equation}\label{eq:4.13}
R^H_{s,t}=\mathcal T^F_{s,t}(Y)+\nabla F(s,Y_s)R^Y_{s,t}.
\end{equation}
On sets on which the initial values and controlled norms of $(Y,Y')$ are bounded,
\begin{align}
\norm{H'}_{p\text{-var};[s,t]}&\le C\left(|t-s|+\norm{Y}_{p\text{-var};[s,t]}+\norm{Y'}_{p\text{-var};[s,t]}\right),\label{eq:4.2}\\
\norm{R^H}_{q\text{-var};[s,t]}&\le C\left(\omega_Y(s,t)^{1/q}+\norm{R^Y}_{q\text{-var};[s,t]}\right).\label{eq:4.3}
\end{align}
\end{theorem}

\begin{proof}
The identity follows from
\begin{align*}
R^H_{s,t}&=F(t,Y_t)-F(s,Y_s)-\nabla F(s,Y_s)Y'_sX_{s,t}\\
&=\bigl(F(t,Y_t)-F(s,Y_s)-\nabla F(s,Y_s)Y_{s,t}\bigr)+\nabla F(s,Y_s)R^Y_{s,t}\\
&=\mathcal T^F_{s,t}(Y)+\nabla F(s,Y_s)R^Y_{s,t}.
\end{align*}
The first term has finite $q$-variation by Theorem~\ref{thm:1}; the second does because $\nabla F$ is bounded.  This proves \eqref{eq:4.3}.  Since $F$ and $\nabla F$ are stopped path Lipschitz, $F(\cdot,Y)$ and $\nabla F(\cdot,Y)$ have finite $p$-variation controlled by $|t-s|+\norm{Y}_{p\text{-var};[s,t]}$.  The product rule for variation gives \eqref{eq:4.2}.
\end{proof}

The improved threshold is now immediate.

\begin{coro}\label{coro:1}
Under the assumptions of Theorem~\ref{thm:2}, the rough integral
\[
\int_0^\cdot F(s,Y_s)\, d\X_s
\]
is well defined for every $2\le p<3$.
\end{coro}

\begin{proof}
Here $q=p/2$, and therefore
\[
\frac1p+\frac1q=\frac3p>1.\qedhere
\]
Then the proof is completed by Lemma \ref{lem:1}.
\end{proof}

We next establish the continuity of the controlled composition map in a slightly weaker control H\"older topology. 
For a continuous control $\rho$ and $\gamma>0$, we write
\[
\|H\|_{\gamma;\rho}:=\sup_{0\le s<t\le T}\frac{|H_{s,t}|}{\rho(s,t)^\gamma}
\]
for a path $H$, and
\[
\|R\|_{\gamma;\rho}:=\sup_{0\le s<t\le T}\frac{|R_{s,t}|}{\rho(s,t)^\gamma}
\]
for a two-parameter map $R$. 

\begin{proposition}\label{prop:4}
Let $\rho$ be a continuous control such that
\begin{equation}\label{eq:4.4}
|t-s|+\norm{X}_{p\text{-var};[s,t]}^p\le \rho(s,t),\quad 0\le s\le t\le T
\end{equation}
and set $\alpha=1/p$, $\beta=2/p$. Suppose that controlled pairs $(Y^n,(Y^n)')$ and $(Y,Y')$ are $(p,q)$-controlled paths over $X$ and satisfy, for some $K<\infty$,
\begin{equation}\label{eq:4.5}
\sup_n\left(\norm{Y^n}_\infty+\norm{(Y^n)'}_\infty+\|(Y^n)'\|_{\alpha;\rho}+\|R^{Y^n}\|_{\beta;\rho}\right)+\norm Y_\infty+\norm{Y'}_\infty+\|Y'\|_{\alpha;\rho}+\|R^Y\|_{\beta;\rho}\le K.
\end{equation}
Assume moreover that $F$ satisfies Assumption~\ref{assum:4} and
\begin{equation}\label{eq:4.6}
\norm{Y^n-Y}_\infty+\norm{(Y^n)'-Y'}_\infty+\norm{R^{Y^n}-R^Y}_\infty\to 0.
\end{equation}
Let $H^n=F(\cdot,Y^n)$, $(H^n)'=\nabla F(\cdot,Y^n)(Y^n)'$ and define $(H,H')$ analogously. Then, for every $\alpha_1\in(0,\alpha)$ and $\beta_1\in(0,\beta)$,
\begin{equation}\label{eq:4.7}
\norm{H^n-H}_\infty+\norm{(H^n)'-H'}_\infty+\|(H^n)'-H'\|_{\alpha_1;\rho}+\|R^{H^n}-R^H\|_{\beta_1;\rho}\to 0.
\end{equation}
\end{proposition}

Before we turn to the proof of this result, we show the following:

\begin{lemma}\label{lem:6}
Let $\rho$ be a continuous control satisfying $\rho(s,t)\ge t-s$, and let $0<\gamma_0<\gamma\le1$.
Suppose that $A^n$ and $A$ are either both one-parameter paths or both two-parameter maps.
Assume
\begin{equation*}
\norm{A^n-A}_\infty\longrightarrow0,\quad\sup_n\bigl(\|A^n\|_{\gamma;\rho}+\|A\|_{\gamma;\rho}\bigr)\le K.
\end{equation*}
Then $\|A^n-A\|_{\gamma_0;\rho}\to0$. More precisely, for every $\varepsilon\in(0,\rho(0,T)]$,
\begin{equation}\label{eq:4.10}
\|A^n-A\|_{\gamma_0;\rho}\le 2K\varepsilon^{\gamma-\gamma_0}+2\norm{A^n-A}_\infty\varepsilon^{-\gamma_0}.
\end{equation}
\end{lemma}

\begin{proof}
If $\rho(s,t)\le\varepsilon$, use the strong control H\"older bound:
\[
\frac{|(A^n-A)_{s,t}|}{\rho(s,t)^{\gamma_0}}\le 2K\rho(s,t)^{\gamma-\gamma_0}\le2K\varepsilon^{\gamma-\gamma_0}.
\]
If $\rho(s,t)>\varepsilon$, use uniform convergence:
\[
\frac{|(A^n-A)_{s,t}|}{\rho(s,t)^{\gamma_0}}\le2\norm{A^n-A}_\infty\varepsilon^{-\gamma_0}.
\]
Taking the supremum gives \eqref{eq:4.10}. 
\end{proof}

With this conclusion in hand, we can now turn to the following proof:

\begin{proof}[Proof of Proposition~\ref{prop:4}]
The equations  \eqref{eq:2.7},  \eqref{eq:4.4} and \eqref{eq:4.5} imply
\begin{equation}\label{eq:4.9}
|Y^n_{s,t}|+|Y_{s,t}|\le C_K\, \rho(s,t)^\alpha,
\end{equation}
hence, for all $0\le s\le t\le T$,
\[
d_\infty\bigl((t,Y^n_t),(s,Y^n_s)\bigr)+d_\infty\bigl((t,Y_t),(s,Y_s)\bigr)\le C_K\,\rho(s,t)^\alpha.
\]
Stopped path Lipschitz continuity of $F$ and $\nabla F$, together with \eqref{eq:4.6}, gives
\begin{equation}\label{eq:4.8}
\norm{H^n-H}_\infty+\norm{(H^n)'-H'}_\infty\to 0.
\end{equation}
By the Lipschitz continuity of $\nabla F$ and \eqref{eq:4.5},
\begin{equation}\label{eq:4.12}
\sup_n\|(H^n)'\|_{\alpha;\rho}+\|H'\|_{\alpha;\rho}\le C_K.
\end{equation}

For the controlled remainder, use the exact identity
\[
R^{H^n}_{s,t}=\mathcal T^F_{s,t}(Y^n)+\nabla F(s,Y^n_s)R^{Y^n}_{s,t}.
\]
By
\[
\omega_{Y^n}(s,t)=|t-s|+\norm{Y^n}_{p\text{-var};[s,t]}^p\overset{\eqref{eq:4.4}, \eqref{eq:4.9}}{\le} C_K\rho(s,t),
\]
the global
$p/2$-variation estimate in Theorem~\ref{thm:1} gives
\[
\norm{\mathcal T^F(Y^n)}_{p/2\text{-var};[s,t]}\le C_K\omega_{Y^n}(s,t)^{2/p}\le C_K\rho(s,t)^\beta.
\]
Together with the preceding remainder identity and the uniform bound on
$\nabla F$, this yields
\[
\sup_n\|R^{H^n}\|_{\beta;\rho}+\|R^H\|_{\beta;\rho}\le C_K.
\]
Furthermore, each term in the remainder identity converges uniformly. 
The convergence of the causal Taylor term follows from $\norm{Y^n-Y}_\infty\to0$ and the stopped-path Lipschitz continuity of $F$ and $\nabla F$, while the second term follows from the uniform convergence of $R^{Y^n}$ and the same Lipschitz property. Therefore,
\begin{equation}\label{eq:4.11}
\norm{R^{H^n}-R^H}_\infty\longrightarrow0.
\end{equation}
Applying Lemma~\ref{lem:6} to $(H^n)'-H'$ using \eqref{eq:4.8} and \eqref{eq:4.12}, and to $R^{H^n}-R^H$ using the strong $\beta$ bounds and \eqref{eq:4.11}, proves \eqref{eq:4.7}.
\end{proof}

\section{Path-dependent rough differential equations}\label{sec:5}

In this section, we establish the existence of solutions to path-dependent rough differential equations using Schauder fixed point theorem.

Let $ \mathbf{X}=(X,\mathbb{X})\in\mathcal{C}^{p\text{-var}}([0,T],\mathbb{R}^{d})$ be a given rough path. On the interval $[t_0,T]$, we consider the equation
\begin{equation}\label{eq:5.1}
\begin{cases}
dY(s)=b(s,Y_s)\,ds+\sigma(s,Y_s)\,dX(s),\\[2mm]
Y_{t_0}=\xi_{t_0},
\end{cases}
\end{equation}
where $b:\Lambda_T^k\longrightarrow\mathbb{R}^{k}$ and $\sigma:\Lambda_T^k\to\mathcal{L}(\mathbb{R}^{d},\mathbb{R}^{k})$ are non-anticipative functionals. 
The results established in the previous section provide the appropriate framework for interpreting the rough term in \eqref{eq:5.1}. Suppose that $\sigma$ satisfies the conditions of Theorem~\ref{thm:2}. 
Then, for every controlled path $(Y,Y')\in\mathcal{D}_{X}^{p,q}([0,T],\mathbb{R}^{k})$, the composition of $\sigma$ with $Y$ is again controlled by $X$, namely,
\begin{equation*}
\Big(\sigma(\cdot,Y), \nabla\sigma(\cdot,Y)Y'\Big) \in\mathcal{D}_{X}^{p,q} \left([0,T], \mathcal{L}(\mathbb{R}^{d},\mathbb{R}^{k})\right).
\end{equation*}
More precisely, we have the following definition.

\begin{definition}\label{def:9}  
Let $p\in [2, 3)$, $q=p/2$ and $\mathbf{X}=(X,\mathbb{X})\in\mathcal{C}^{p\text{-var}}([0,T],\mathbb{R}^d)$. Assume 
\[
b\in C_b^{0,0}(\Lambda_T^k,\mathbb{R}^k),\quad \sigma\in C_b^{0,1}\big(\Lambda_T^k,\mathcal{L}(\mathbb{R}^{d},\mathbb{R}^{k})\big)
\]
with $\sigma$ and $\nabla\sigma$ locally horizontally Lipschitz continuous and belonging to $\operatorname{Lip}(\Lambda_T^k,\|\cdot\|_\infty)$. 
Given $\xi\in C^{p\text{-var}}([0,t_0],\mathbb{R}^k)$, a pair $(Y,Y')$ is called a \emph{solution} to \eqref{eq:5.1} if $Y|_{[0,t_0]}=\xi$,
\[
(Y, Y')\in \mathcal{D}_X^{p,q}([t_0,T],\mathbb{R}^k),
\]
$Y'_t=\sigma(t,Y_t)$ for $t\in[0,T]$, and
\begin{equation*}
Y_t=\xi_{t_0}+\int_{t_0}^t b(s,Y_s)\, d s+\int_{t_0}^t\sigma(s,Y_s)\, d\X_s, \quad t\in[t_0,T],
\end{equation*}
where the rough integral uses the controlled path $\big(\sigma(\cdot,Y),\nabla\sigma(\cdot,Y)Y'\big)$ on $[t_0,T]$.
\end{definition}

We adopt the same assumptions as in \cite{An23}.
Let
\[
b:\Lambda_T^k\to\RR^k,\quad \sigma:\Lambda_T^k\to\cL(\RR^d,\RR^k).
\]

\begin{assumption}\cite[Assumption 4.2]{An23}\label{assum:5}
The functional $b$ is globally Lipschitz in $d_\infty$.
\end{assumption}

\begin{assumption}\cite[Assumption 4.3]{An23}\label{assum:6}
The functional $\sigma$ satisfies Assumption~\ref{assum:4}.
\end{assumption}

To establish the existence of solutions to \eqref{eq:5.1} for $2\le p<3$, we use a Schauder fixed point argument~\cite{GT01},
following the approach of Ananova~\cite{An23}.
Choose the rough path control
\begin{equation}\label{eq:5.4}
\rho(s,t):=|t-s|+\norm{X}_{p\text{-var};[s,t]}^p+\norm{\x}_{p/2\text{-var};[s,t]}^{p/2}.
\end{equation}
For the fixed point problem started at time zero, set
\begin{equation}\label{eq:5.2}
y_0:=\xi_0,\quad y'_0:=\sigma(0,\xi_0).
\end{equation}
Set $\alpha:=1/p$ and $\beta:=2/p$. Fix $0<\alpha_0<\alpha$ and $0<\beta_0<\beta$.  
Let 
\begin{equation}\label{eq:dxrho}
\mathcal{D}_{X,\rho}^{\alpha_0,\beta_0}([0,T],\mathbb R^k)
\end{equation}
denote the space of controlled paths $(Y,Y')$  on $[0,T]$ satisfying
\[
Y_{s,t}=Y'_sX_{s,t}+R^Y_{s,t},\quad \|Y'\|_{\alpha_0;\rho}+\|R^Y\|_{\beta_0;\rho}<\infty,
\]
equipped with the norm
\[
\norm{Y,Y'}_{\mathcal{D}_{X,\rho}^{\alpha_0,\beta_0}}:=\norm Y_\infty+\norm{Y'}_\infty+\|Y'\|_{\alpha_0;\rho}+\|R^Y\|_{\beta_0;\rho}.
\]
The stronger $(\alpha,\beta)$ bounds yield compactness in the space $\mathcal{D}_{X,\rho}^{\alpha_0,\beta_0}([0,T],\mathbb R^k)$, as stated in the following proposition.

\begin{proposition}\label{prop:5}
Fix initial values $y_0,y'_0$ and $M>0$.  The set
\begin{equation}\label{eq:5.3}
\mathcal A_M:=\left\{(Y,Y')\in \mathcal{D}_{X,\rho}^{\alpha_0,\beta_0}([0,T],\mathbb R^k)\mid Y_0=y_0,\ Y'_0=y'_0, \|Y'\|_{\alpha;\rho}+\|R^Y\|_{\beta;\rho}\le M\right\}
\end{equation}
is convex and compact in $\mathcal{D}_{X,\rho}^{\alpha_0,\beta_0}([0,T],\mathbb R^k)$.
\end{proposition}

\begin{proof}
The set is nonempty: the  pair
\[
U_t:=y_0+y'_0X_{0,t},\quad U'_t:=y'_0
\]
has $R^{U}=0$ and belongs to $\mathcal A_M$ for every $M>0$.

To prove convexity, let $\lambda\in [0, 1]$ and $(Y,Y'), (\widetilde{Y}, \widetilde{Y}')\in \mathcal A_M$. Define 
\[
(Z, Z')=\lambda (Y,Y')+(1-\lambda)(\widetilde{Y}, \widetilde{Y}').
\]
Since the controlled relation is linear, the corresponding remainder satisfies
\[
R^Z=\lambda R^Y+(1-\lambda)R^{\widetilde Y}.
\]
Therefore
\[
\|Z'\|_{\alpha;\rho}+\|R^Z\|_{\beta;\rho}\le\lambda M+(1-\lambda)M=M.
\]
Moreover, the prescribed initial values are preserved under convex combinations. Hence
\[
(Z,Z')\in\mathcal A_M,
\]
and thus $\mathcal A_M$ is convex.

It remains to prove compactness. Let $\{(Y^n,(Y^n)')\}_{n\ge1}\subset\mathcal A_M$. Since the initial values are fixed and $\|(Y^n)'\|_{\alpha;\rho}\le M$, we have
\[
|(Y^n)'_{s,t}|\le M\rho(s,t)^\alpha, \quad 
|(Y^n)'_{t}|\le |y'_0|+M\rho(0,t)^\alpha\le |y'_0|+M\rho(0,T)^\alpha,
\]
where $0\le s\le t\le T$. Taking the supremum of $n$,
\[
\sup_{n\ge1}\norm{(Y^n)'}_\infty\le |y'_0|+M\rho(0,T)^\alpha<\infty.
\]
Since $\rho$ is continuous, the family $\{(Y^n)'\}_{n\ge1}$ is uniformly bounded and equicontinuous. Using the controlled relation
\[
Y^n_{s,t}=(Y^n)'_sX_{s,t}+R^{Y^n}_{s,t},
\]
together with $\|R^{Y^n}\|_{\beta;\rho}\le M$, we obtain
\[
|Y^n_{s,t}|\le\norm{(Y^n)'}_\infty |X_{s,t}|+M\rho(s,t)^\beta.
\]
Since $X$ is continuous and $\rho$ is a continuous control, the family $\{Y^n\}_{n\ge1}$ is also uniformly bounded and equicontinuous. By the Arzel\`a--Ascoli theorem, there exists a subsequence, still denoted by $\{(Y^n,(Y^n)')\}$, and continuous paths $Y$ and $Y'$ such that
\[
\norm{Y^n-Y}_\infty+\norm{(Y^n)'-Y'}_\infty \to 0.
\]
Since $X$ is fixed and continuous,
\[
\big|R^{Y^n}_{s,t}-R^Y_{s,t}\big|=\big|\big(Y_{s,t}-Y'_sX_{s,t}\big)-\big(Y^n_{s,t}-(Y^n)'_sX_{s,t}\big)\big|\le
|Y^n_{s,t}-Y_{s,t}|+|(Y^n)'_s-Y'_s|\,|X_{s,t}|
\]
and hence
\[
\norm{R^{Y^n}-R^Y}_\infty\to 0.
\]
The  bounds
\[
\sup_{n\ge1}\|(Y^n)'\|_{\alpha;\rho}\le M,\quad \sup_{n\ge1}\|R^{Y^n}\|_{\beta;\rho}\le M,
\]
together with the uniform convergences above and Lemma~\ref{lem:6}, imply
\[
\|(Y^n)'-Y'\|_{\alpha_0;\rho}\to 0,\quad \|R^{Y^n}-R^Y\|_{\beta_0;\rho}\to 0.
\]
Consequently, $$(Y^n,(Y^n)')\to (Y,Y') \,\text{ in }\, \mathcal{D}_{X,\rho}^{\alpha_0,\beta_0}([0,T],\mathbb R^k).$$ 

Now we verify that the limit belongs to $\mathcal A_M$. Since $(Y^n)'\to Y'$ uniformly, for every fixed $0\le s<t\le T$, $$(Y^n)'_{s,t}\to Y'_{s,t}.$$ Hence
\[
\frac{|Y'_{s,t}|}{\rho(s,t)^\alpha}=\lim_{n\to\infty}\frac{|(Y^n)'_{s,t}|}{\rho(s,t)^\alpha}.
\]
By the definition of the H\"older seminorm,
\[
\frac{|(Y^n)'_{s,t}|}{\rho(s,t)^\alpha}\le \|(Y^n)'\|_{\alpha;\rho}
\]
for every $n$. Therefore,
\[
\frac{|Y'_{s,t}|}{\rho(s,t)^\alpha}\le\liminf_{n\to\infty}\|(Y^n)'\|_{\alpha;\rho}.
\]
Taking the supremum over $0\le s<t\le T$ yields
\[
\|Y'\|_{\alpha;\rho}=\sup_{0\le s<t\le T}\frac{|Y'_{s,t}|}{\rho(s,t)^\alpha}\le\liminf_{n\to\infty}\|(Y^n)'\|_{\alpha;\rho}.
\]
The same argument applies to the remainders. Since $R^{Y^n}\to R^Y$ uniformly, we have $R^{Y^n}_{s,t}\to R^Y_{s,t}$ for every fixed $0\le s<t\le T$. Consequently,
\[
\|R^Y\|_{\beta;\rho}\le\liminf_{n\to\infty}\|R^{Y^n}\|_{\beta;\rho}.
\]
Therefore,
\[
\|Y'\|_{\alpha;\rho}+\|R^Y\|_{\beta;\rho}\le\liminf_{n\to\infty}\left(\|(Y^n)'\|_{\alpha;\rho}+\|R^{Y^n}\|_{\beta;\rho}\right)\le M.
\]
Thus $(Y,Y')\in\mathcal A_M$. Hence $\mathcal A_M$ is compact in $\mathcal{D}_{X,\rho}^{\alpha_0,\beta_0}([0,T],\mathbb R^k)$.
\end{proof}

According to Definition~\ref{def:9}, we defined a map
\[
\mathcal M: \mathcal A_M\to \mathcal{D}_{X,\rho}^{\alpha_0,\beta_0}([0,T],\mathbb R^k), \quad (Y,Y')\mapsto\mathcal M(Y,Y'):=(Z,Z')
\]
by
\begin{align*}
Z_t:=\xi_0+\int_0^t b(s,Y_s)\,d s+\int_0^t\sigma(s,Y_s)\,d\X_s,\quad Z'_t:=\sigma(t,Y_t).
\end{align*}
The map is well defined on $\mathcal A_M$ by Theorem~\ref{thm:2} and Corollary~\ref{coro:1}. We prove $\mathcal M$ to map $\mathcal A_M$ into $\mathcal A_M$.

\begin{proposition}\label{prop:6}
Assume $2\le p<3$ and fix the initial values $y_0,y'_0$ given by \eqref{eq:5.2}.  There exist $M<\infty$ and
$\delta\in (0, 1]$ such that, on every interval $[0,T_0]$ with $\rho(0,T_0)\le\delta$,
\[
 \mathcal M(\mathcal A_M)\subset\mathcal A_M.
\]
\end{proposition}

\begin{proof}
By \eqref{eq:5.3}, we need to prove $$\|Z'\|_{\alpha;\rho}+\|R^Z\|_{\beta;\rho}\le M.$$ Since $Y'_0=y'_0$,
\begin{equation}\label{eq:5.5}
\norm{Y'}_\infty\le |y'_0|+\sup_{s\in [0, T_0]}\frac{|Y'_s-Y'_0|}{\rho(0,s)^\alpha}\rho(0,s)^\alpha
\le |y'_0|+\|Y'\|_{\alpha;\rho}\rho(0,T_0)^\alpha
\le |y'_0|+M\delta^\alpha.
\end{equation}
By the controlled identity,
\[
|Y_{s, t}|\le |Y'_s|\,|X_{s,t}|+|R^Y_{s,t}|\overset{\eqref{eq:5.4}, \eqref{eq:5.3}}{\le} \norm{Y'}_\infty \rho(s,t)^\alpha+M  \rho(s,t)^\beta.
\]
Divide both sides by $\rho(s,t)^\alpha$ to obtain 
\begin{equation}\label{eq:5.6}
\|Y\|_{\alpha;\rho}=\sup_{0\le s<t\le T_0}\frac{|Y_{s, t}|}{\rho(s,t)^\alpha}\le \norm{Y'}_\infty +M  \delta^\alpha\overset{\eqref{eq:5.5}}{\le} |y'_0|+2M\delta^\alpha=:K_Y.
\end{equation}
Consequently
\begin{equation}\label{eq:5.9}
\norm Y_\infty\le |y_0|+K_Y\delta^\alpha,
\quad
\norm{Y}_{p\text{-var};[s,t]}^p\le K_Y^p\rho(s,t).
\end{equation}
For stopped paths,
\begin{equation}\label{eq:5.7}
\begin{aligned}
d_\infty\big((s,Y_s),(t,Y_t)\big)=&\  |t-s|+\|Y_t-Y_s\|_\infty \\
=&\ |t-s|+\sup_{r\in [s, t]}|Y_{s, r}|\\
\le&\ \rho(s, t)+K_Y\rho(s, t)^\alpha \hspace{1cm} (\text{by (\ref{eq:5.4}) and (\ref{eq:5.6})})\\
=&\ \big(\rho(s, t)^{1-\alpha}+K_Y\big)\rho(s, t)^\alpha\\
\le&\ \big(\rho(0, T_0)^{1-\alpha}+K_Y\big)\rho(s, t)^\alpha\\
\le&\ (\delta^{1-\alpha}+K_Y)\rho(s,t)^\alpha.
\end{aligned}
\end{equation}
The Lipschitz continuity of $\sigma$ therefore gives
\begin{equation}\label{eq:5.8}
\begin{aligned}
\|Z'\|_{\alpha;\rho}=&\  \sup_{0\le s<t\le T_0}\frac{|\sigma(t,Y_t)-\sigma(s,Y_s)|}{\rho(s,t)^\alpha}\\
\le&\ \sup_{0\le s<t\le T_0}\frac{L_{\sigma}d_\infty\big((s,Y_s),(t,Y_t)\big)}{\rho(s,t)^\alpha}\\
\le&\ L_{\sigma}(\delta^{1-\alpha}+K_Y) \hspace{1cm} (\text{by (\ref{eq:5.7})})\\
=&\ L_{\sigma}(\delta^{1-\alpha}+|y'_0|+2M\delta^\alpha) \hspace{1cm} (\text{by (\ref{eq:5.6})})\\
\le&\ L_{\sigma}(1+|y'_0|)+2L_{\sigma}M\delta^\alpha \hspace{1cm} (\text{by $\delta\le 1$})\\
\le&\  C_0+C_1M\delta^\alpha.
\end{aligned}
\end{equation}
Here and below, $C_0,C_1,C_2\ge 1$ denote sufficiently large constants, which may be enlarged when necessary and depend only on the coefficients and the initial history, but not on $M$ or $\delta$. Put $H=\sigma(\cdot,Y)$ and $H'=\nabla\sigma(\cdot,Y)Y'$. By the global Lipschitz continuity of $\sigma$,
\begin{equation}\label{eq:5.19}
\norm{H'}_\infty\le C\norm{Y'}_\infty\overset{\eqref{eq:5.5}}{\le} C_0+C_1M\delta^\alpha.
\end{equation}
Since 
\[
H'_{s, t}=\nabla\sigma(t,Y_t)Y'_t-\nabla\sigma(s,Y_s)Y'_s=\big(\nabla\sigma(t,Y_t)-\nabla\sigma(s,Y_s)\big)Y'_t+\nabla\sigma(s,Y_s)(Y'_t-Y'_s),
\]
we have
\[
\|H'\|_{\alpha;\rho}\le \|\nabla\sigma(\cdot,Y)\|_{\alpha;\rho}\|Y'\|_{\infty}+\|\nabla\sigma(\cdot,Y)\|_{\infty}\|Y'\|_{\alpha;\rho}.
\]
By Assumption~\ref{assum:6}, $\nabla\sigma$ is globally bounded. \eqref{eq:5.5} and \eqref{eq:5.7} give
\begin{equation}\label{eq:5.20}
\|H'\|_{\alpha;\rho}\le C_0+C_1M+C_2M^2\delta^\alpha.
\end{equation}
Furthermore, by \eqref{eq:5.9},
\begin{equation}\label{eq:5.10}
\omega_Y(s,t):=|t-s|+\norm{Y}_{p\text{-var};[s,t]}^p\le(1+K_Y^p)\rho(s,t).
\end{equation}
Decrease $\delta$ so that $(1+K_Y^p)\delta\le1$. Then
\begin{equation}\label{eq:5.11}
\begin{aligned}
|R^H_{s, t}|=&\ |\mathcal T^\sigma_{s,t}(Y)+\nabla \sigma(s,Y_s)R^Y_{s,t}| \hspace{1cm} (\text{by (\ref{eq:4.13})})\\
\le&\ |\mathcal T^\sigma_{s,t}(Y)|+|\nabla \sigma(s,Y_s)R^Y_{s,t}|\\
\le&\ C \omega_Y(s,t)^{2/p}+CM\rho(s,t)^\beta\hspace{1cm} (\text{by Theorem~\ref{thm:1} and $\|R^Y\|_{\beta;\rho}\le M$})\\
\le&\ C\Big((1+K_Y^p)\rho(s,t)\Big)^{2/p}+CM\rho(s,t)^\beta\hspace{1cm} (\text{by (\ref{eq:5.10})})\\
=&\ C(1+K_Y^p)^{2/p}\rho(s,t)^\beta+CM\rho(s,t)^\beta\\
\le&\ C(1+K_Y^2)\rho(s,t)^\beta+CM\rho(s,t)^\beta.
\end{aligned}
\end{equation}
The last step involved the use of $$(1+a^p)^{2/p}\le C_p(1+a^2).$$ 
Divide both sides by $\rho(s,t)^\beta$ to obtain 
\begin{equation}\label{eq:5.12}
\|R^H\|_{\beta;\rho}\le C\bigl(1+K_Y^2+M\bigr)\overset{\eqref{eq:5.6}}{\le} C_0+C_1M+C_2M^2\delta^{2\alpha}.
\end{equation}
The controlled remainder of $Z$ is
\begin{equation}\label{eq:5.13}
\begin{aligned}
R^Z_{s,t}=&\ \int_s^t b(r,Y_r)d r+\int_s^tHd\X-H_sX_{s, t}\\
=&\ \int_s^t b(r,Y_r)dr +H'_s\x_{s,t}+\left(\int_s^tHd\X-H_sX_{s,t}-H'_s\x_{s,t}\right).
\end{aligned}
\end{equation}
We estimate each term on the right side of the equation separately. For $\int_s^t b(r,Y_r)dr$,
\begin{equation}\label{eq:5.14}
\Big|\int_s^t b(r,Y_r)dr\Big|\le \|b(\cdot,Y)\|_{\infty}|t-s|\overset{\eqref{eq:5.4}}{\le}\|b(\cdot,Y)\|_{\infty}\rho(s,t)\le\big(\|b(\cdot,Y)\|_{\infty}\delta^{1-\beta}\big)\rho(s,t)^\beta.
\end{equation}
For $H'_s\x_{s,t}$,
\begin{equation}\label{eq:5.15}
|H'_s\x_{s,t}|\le \|H'\|_{\infty}|\x_{s,t}|\overset{\eqref{eq:5.4}}{\le} \|H'\|_{\infty}\rho(s,t)^\beta.
\end{equation}
For the last term,
\begin{equation}\label{eq:5.16}
\begin{aligned}
&\ \Big|\int_s^tH\,d\X-H_sX_{s,t}-H'_s\x_{s,t}\Big|\\
\le&\  C\Big( \|R^H\|_{q\text{-var};[s,t]} \|X\|_{p\text{-var};[s,t]}+\|H'\|_{p\text{-var};[s,t]}\|\mathbb X\|_{\frac p2\text{-var};[s,t]}\Big)\hspace{1cm} (\text{by Lemma~\ref{lem:1}})\\
\le&\ C\Big( \|R^H\|_{\beta; \rho}\rho(s,t)^\beta \|X\|_{p\text{-var};[s,t]}+\|H'\|_{\alpha; \rho}\rho(s,t)^\alpha\|\mathbb X\|_{\frac p2\text{-var};[s,t]}\Big)\ \  (\text{by Lemma~\ref{lem:5}})\\
\le&\ C \|R^H\|_{\beta; \rho}\rho(s,t)^{\beta+\alpha} +C\|H'\|_{\alpha; \rho}\rho(s,t)^{\alpha+\beta}\hspace{1cm} (\text{by (\ref{eq:5.4})})\\
\le&\  C \|R^H\|_{\beta; \rho}\delta^\alpha\rho(s,t)^\beta+C\|H'\|_{\alpha; \rho}\delta^\alpha\rho(s,t)^\beta.
\end{aligned}
\end{equation}
Substituting \eqref{eq:5.14}, \eqref{eq:5.15} and \eqref{eq:5.16} into \eqref{eq:5.13}, 
\begin{equation}\label{eq:5.17}
\|R^Z\|_{\beta;\rho}\le \norm{b(\cdot,Y)}_\infty\delta^{1-\beta}+\norm{H'}_\infty+C\|R^H\|_{\beta;\rho}\delta^\alpha
+C\|H'\|_{\alpha;\rho}\delta^\alpha.
\end{equation}
Since $b$ is globally Lipschitz,
\begin{equation}\label{eq:5.18}
\begin{aligned}
\norm{b(\cdot,Y)}_\infty=&\  \sup_{t\in[0, T_0]}|b(t,Y_t)|\\
\le&\ |b(0, y_0)|+L_bd_\infty\big((0, Y_0),(t,Y_t)\big)\\
\le&\ |b(0, y_0)|+L_b(\delta^{1-\alpha}+K_Y)\rho(0,t)^\alpha\hspace{1cm} (\text{by (\ref{eq:5.7})})\\
\le&\ |b(0, y_0)|+L_b\delta+L_bK_Y\delta^\alpha\\
\le&\ |b(0, y_0)|+L_b\delta+L_b|y'_0|\delta^\alpha+2L_bM\delta^{2\alpha}\hspace{1cm} (\text{by (\ref{eq:5.6})})\\
\le&\ C_0+C_1M\delta^{2\alpha}.
\end{aligned}
\end{equation}
Substituting \eqref{eq:5.18}, \eqref{eq:5.19}, \eqref{eq:5.12} and \eqref{eq:5.20} into \eqref{eq:5.17},
\begin{equation}\label{eq:5.21}
\begin{aligned}
\|R^Z\|_{\beta;\rho}\le&\  (C_0+C_1M\delta^{2\alpha})\delta^{1-\beta}+(C_0+C_1M\delta^\alpha)+C(C_0+C_1M+C_2M^2\delta^{2\alpha})\delta^{\alpha}\\
&\ +C(C_0+C_1M+C_2M^2\delta^\alpha)\delta^\alpha\\
=&\  C_0\delta^{1-\beta}+C_1M\delta+(C_0+C_1M\delta^\alpha)+C(C_0+C_1M+C_2M^2\delta^{2\alpha})\delta^{\alpha}\\
&\ +C(C_0+C_1M+C_2M^2\delta^\alpha)\delta^\alpha \hspace{1cm} (\text{by $2\alpha=\beta$}).
\end{aligned}
\end{equation}
If $p>2$, then $\beta<1$, so $1-\beta>0$, and both terms above contain a positive power of $\delta$. If $p=2$, then $\beta=1$, so the first term reduces to $C_0$. 

Combining \eqref{eq:5.8} and \eqref{eq:5.21}, 
\begin{equation*}
\begin{aligned}
\|Z'\|_{\alpha;\rho}+\|R^Z\|_{\beta;\rho}\le&\ C_0+C_1M\delta^\alpha+ C_0\delta^{1-\beta}+C_1M\delta+(C_0+C_1M\delta^\alpha)\\
&\ +C(C_0+C_1M+C_2M^2\delta^{2\alpha})\delta^{\alpha}+C(C_0+C_1M+C_2M^2\delta^\alpha)\delta^\alpha \\
\le&\ 3C_0+C_1M\delta^\alpha+C_1M\delta+C_1M\delta^\alpha+C(C_0+C_1M+C_2M^2\delta^{2\alpha})\delta^{\alpha}\\
&\ +C(C_0+C_1M+C_2M^2\delta^\alpha)\delta^\alpha.
\end{aligned}
\end{equation*}
Setting $M=4C_0$, for sufficiently small $\delta$, we have $$\|Z'\|_{\alpha;\rho}+\|R^Z\|_{\beta;\rho}\le M.$$ 
Hence $\mathcal M(\mathcal A_M)\subset\mathcal A_M$, as required. 
\end{proof}

It remains to establish the continuity of the mapping $\mathcal M$.
Recall that the $\mathcal{D}_{X,\rho}^{\alpha_0,\beta_0}([0,T_0],\mathbb R^k)$ is given in (\ref{eq:dxrho}).

\begin{proposition}\label{prop:7}
Choose $\alpha_0<\alpha$ and $\beta_0<\beta$ such that
\begin{equation}\label{eq:5.22}
\alpha_0+\beta>1,\quad \beta_0+\alpha>1.
\end{equation}
Then $\mathcal M:\mathcal A_M\to\mathcal A_M$ is continuous in $\mathcal{D}_{X,\rho}^{\alpha_0,\beta_0}([0,T_0],\mathbb R^k)$.
\end{proposition}

\begin{proof}
Let $$\big\{\big(Y^n,(Y^n)'\big)\big\}_n,\, (Y,Y')\in\mathcal A_M.$$ Suppose 
$$\big(Y^n,(Y^n)'\big)\to(Y,Y')\,\text{ in }\, \mathcal{D}_{X,\rho}^{\alpha_0,\beta_0}([0,T_0],\mathbb R^k).$$ This means that
\[
\sup_n\left(\norm{Y^n}_\infty+\norm{(Y^n)'}_\infty+\|(Y^n)'\|_{\alpha;\rho}+\|R^{Y^n}\|_{\beta;\rho}\right)+\norm Y_\infty+\norm{Y'}_\infty+\|Y'\|_{\alpha;\rho}+\|R^Y\|_{\beta;\rho}\le C_M
\]
and
\[
\norm{Y^n-Y}_\infty+\norm{(Y^n)'-Y'}_\infty+\norm{R^{Y^n}-R^Y}_\infty\to 0.
\]
Choose intermediate exponents
\[
\alpha_0<\alpha_1<\alpha,\quad \beta_0<\beta_1<\beta.
\]
Then  $$\alpha_1+\beta>1, \quad \beta_1+\alpha>1.$$  By Proposition~\ref{prop:4}, for
\[
H^n=\sigma(\cdot,Y^n),\quad (H^n)'=\nabla\sigma(\cdot,Y^n)(Y^n)'
\]
and the analogous pair $(H,H')$, one has
\begin{equation}\label{eq:5.23}
\norm{H^n-H}_\infty+\norm{(H^n)'-H'}_\infty+\|(H^n)'-H'\|_{\alpha_1;\rho}+\|R^{H^n}-R^H\|_{\beta_1;\rho}\to 0.
\end{equation}

Define
\[
U^n_t :=\int_0^tH^n_r\,d\X_r,\quad U_t :=\int_0^tH_r\,d\X_r.
\]
Setting
\[
A^n_{s,t}=H^n_sX_{s,t}+(H^n)'_s\x_{s,t},\quad A_{s,t}=H_sX_{s,t}+H'_s\x_{s,t},
\]
Chen's relation gives
\begin{equation}\label{eq:5.24}
(A^n-A)_{s,t}-(A^n-A)_{s,u}-(A^n-A)_{u,t}=-(R^{H^n}-R^H)_{s,u}X_{u,t}-\bigl((H^n)'-H'\bigr)_{s,u}\x_{u,t},
\end{equation}
where $0\le s \le u \le t \le T_0$. Since
\begin{align*}
|(R^{H^n}-R^H)_{s,u}X_{u,t}|\le&\  |(R^{H^n}-R^H)_{s,u}|\,|X_{u,t}|\\
\le&\ \|R^{H^n}-R^H\|_{\beta_1; \rho}\rho(s, u)^{\beta_1}\rho(u, t)^{\alpha}\hspace{1cm} (\text{by (\ref{eq:5.4})})\\
\le&\ \|R^{H^n}-R^H\|_{\beta_1; \rho}\rho(s, t)^{\beta_1+\alpha}
\end{align*}
and
\begin{align*}
\big|\bigl((H^n)'-H'\bigr)_{s,u}\x_{u,t}\big|\le&\  \big|\bigl((H^n)'-H'\bigr)_{s,u}\big|\,\big|\x_{u,t}\big|\\
\le&\ \|(H^n)'-H'\|_{\alpha_1; \rho}\rho(s, u)^{\alpha_1}\rho(u, t)^{\beta}\hspace{1cm} (\text{by (\ref{eq:5.4})})\\
\le&\ \|(H^n)'-H'\|_{\alpha_1; \rho}\rho(s, t)^{\alpha_1+\beta},
\end{align*}
we have
\[
|(A^n-A)_{s,t}-(A^n-A)_{s,u}-(A^n-A)_{u,t}|\le \|R^{H^n}-R^H\|_{\beta_1; \rho}\rho(s, t)^{\beta_1+\alpha}+\|(H^n)'-H'\|_{\alpha_1; \rho}\rho(s, t)^{\alpha_1+\beta}.
\]
Note that $\alpha_1+\beta>1$ and $\beta_1+\alpha>1$ exactly meet the conditions for sewing lemma~\cite[Lemma 4.2]{FH20}. The sewing lemma implies, with
$\mathcal E^n_{s,t}:=U^n_{s,t}-A^n_{s,t}$ and
$\mathcal E_{s,t}:=U_{s,t}-A_{s,t}$,
\begin{equation}\label{eq:5.25}
|\mathcal E^n_{s,t}-\mathcal E_{s,t}|\le C\Big(\|R^{H^n}-R^H\|_{\beta_1;\rho} \rho(s,t)^{\beta_1+\alpha}
+\|(H^n)'-H'\|_{\alpha_1;\rho}\rho(s,t)^{\alpha_1+\beta}\Big).
\end{equation}
Since the two exponents on the right are strictly larger than one, and in particular larger than $\beta_0$, \eqref{eq:5.23} gives
\begin{equation}\label{eq:5.26}
\|\mathcal E^n-\mathcal E\|_{\beta_0;\rho}\to 0.
\end{equation}
By Lemma~\ref{lem:1}, $(U^n, H^n)$ is a controlled rough path by $X$. Its controlled remainder is
\[
R^{U^n}_{s,t}=U^n_{s,t}-H^n_sX_{s,t}=(H^n)'_s\x_{s,t}+\big(U^n_{s,t}-H^n_sX_{s,t}-(H^n)'_s\x_{s,t}\big)=(H^n)'_s\x_{s,t}+\mathcal E^n_{s,t}.
\]
Thus, using $|\x_{s,t}|\le\rho(s,t)^\beta$,
\begin{equation}\label{eq:5.29}
\|R^{U^n}-R^U\|_{\beta_0;\rho}\le \norm{(H^n)'-H'}_\infty \rho(0,T_0)^{\beta-\beta_0}+
\|\mathcal E^n-\mathcal E\|_{\beta_0;\rho}\to 0. 
\end{equation}
Since $U^n_0=U_0=0$, we have 
\begin{align*}
U_t^n=&\  H^n_0X_{0,t}+(H^n)'_0\x_{0,t}+\mathcal E^n_{0,t},\\
U_t=&\ H_0X_{0,t}+H'_0\x_{0,t}+\mathcal E_{0,t}.
\end{align*}
Subtracting the two equations gives 

\begin{align*}
|U_t^n-U_t|=&\  |(H^n_0-H_0)X_{0,t}+((H^n)'_0-H'_0)\x_{0,t}+\mathcal E^n_{0,t}-\mathcal E_{0,t}|\\
\le&\ |H^n_0-H_0|\,|X_{0,t}|+|(H^n)'_0-H'_0|\,|\x_{0,t}|+|\mathcal E^n_{0,t}-\mathcal E_{0,t}|\\
\le&\ \norm{H^n-H}_\infty\rho(0,t)^\alpha+\norm{(H^n)'-H'}_\infty\rho(0,t)^\beta+\sup_{0\le t\le T_0}|\mathcal E^n_{0,t}-\mathcal E_{0,t}|\quad (\text{by (\ref{eq:5.4})})\\
\le&\ \norm{H^n-H}_\infty\rho(0,T_0)^\alpha+\norm{(H^n)'-H'}_\infty\rho(0,T_0)^\beta+\sup_{0\le t\le T_0}|\mathcal E^n_{0,t}-\mathcal E_{0,t}|.
\end{align*}
Taking the supremum over $t \in [0, T_0]$ yields 
\begin{equation}\label{eq:5.27}
\begin{aligned}
&\ \norm{U^n-U}_\infty\\
&\ \le \norm{H^n-H}_\infty\rho(0,T_0)^\alpha+\norm{(H^n)'-H'}_\infty\rho(0,T_0)^\beta+\sup_{0\le t\le T_0}|\mathcal E^n_{0,t}-\mathcal E_{0,t}|\to 0.
\end{aligned}
\end{equation}
We also need to prove $$\|H^n-H\|_{\alpha_0;\rho}\to 0.$$ For every $0\le s<t\le T_0$, we have
\[
|(H^n-H)_{s,t}|\le |H^n_t-H_t|+|H^n_s-H_s|\le 2\norm{H^n-H}_\infty,
\]
while, by the definition of the control H\"older seminorm,
\[
|(H^n-H)_{s,t}|\le \|H^n-H\|_{\alpha;\rho}\rho(s,t)^\alpha.
\]
So
\[
|(H^n-H)_{s,t}|\le\min\left\{2\norm{H^n-H}_\infty,\,\|H^n-H\|_{\alpha;\rho}\rho(s,t)^\alpha\right\}.
\]
Set $\theta:=\alpha_0/\alpha\in(0,1)$. Using $\min\{a,b\}\le a^{1-\theta}b^\theta$, we obtain
\begin{align*}
|(H^n-H)_{s,t}|\le&\ \bigl(2\norm{H^n-H}_\infty\bigr)^{1-\theta}\bigl(\|H^n-H\|_{\alpha;\rho}\rho(s,t)^\alpha\bigr)^\theta\\
=&\ 2^{1-\theta}\norm{H^n-H}_\infty^{1-\theta}\|H^n-H\|_{\alpha;\rho}^{\theta}\rho(s,t)^{\alpha\theta}.
\end{align*}
Since $\alpha\theta=\alpha_0$, dividing by $\rho(s,t)^{\alpha_0}$ and taking the supremum over $s<t$ yields
\[
\|H^n-H\|_{\alpha_0;\rho}\le C\norm{H^n-H}_\infty^{1-\alpha_0/\alpha}\|H^n-H\|_{\alpha;\rho}^{\alpha_0/\alpha}.
\]
Therefore, by \eqref{eq:5.23},
\begin{equation}\label{eq:5.28}
\|H^n-H\|_{\alpha_0;\rho}\to 0.
\end{equation}
So, by \eqref{eq:5.27}, \eqref{eq:5.23}, \eqref{eq:5.28} and \eqref{eq:5.29}, we obtain 
$$(U^n,H^n)\to(U,H) \,\text{ in }\, \mathcal{D}_{X,\rho}^{\alpha_0,\beta_0}([0,T_0],\mathbb R^k).$$

It remains to treat $\int_0^t b(s,Y_s)\,d s$. Put
\[
D^n_t=\int_0^t\bigl(b(r,Y^n_r)-b(r,Y_r)\bigr)\,d r.
\]
Stopped-path Lipschitz continuity and uniform convergence of $Y^n$ imply 
\[
\norm{b(\cdot,Y^n)-b(\cdot,Y)}_\infty\longrightarrow0.
\]
Then
\[
\norm{D^n}_\infty\le T_0\norm{b(\cdot,Y^n)-b(\cdot,Y)}_\infty\to 0.
\]
Moreover,
\begin{align*}
\|D^n\|_{\beta_0;\rho}=&\  \sup_{0\le s<t\le T_0}\frac{|D^n_{s, t}|}{\rho(s, t)^{\beta_0}}\\
\le&\ \sup_{0\le s<t\le T_0}\frac{\norm{b(\cdot,Y^n)-b(\cdot,Y)}_\infty|t-s|}{\rho(s, t)^{\beta_0}}\\
\le&\ \sup_{0\le s<t\le T_0}\frac{T_0^{1-\beta_0}\norm{b(\cdot,Y^n)-b(\cdot,Y)}_\infty|t-s|^{\beta_0}}{\rho(s, t)^{\beta_0}}\\
\le&\ \sup_{0\le s<t\le T_0}\frac{T_0^{1-\beta_0}\norm{b(\cdot,Y^n)-b(\cdot,Y)}_\infty\rho(s, t)^{\beta_0}}{\rho(s, t)^{\beta_0}} \hspace{1cm} (\text{by (\ref{eq:5.4})})\\
=&\ T_0^{1-\beta_0}\norm{b(\cdot,Y^n)-b(\cdot,Y)}_\infty\to 0.
\end{align*}
Hence $$\mathcal M(Y^n,(Y^n)')\to \mathcal M(Y,Y')\,\text{ in }\, \mathcal{D}_{X,\rho}^{\alpha_0,\beta_0}([0,T_0],\mathbb R^k),$$ that is, $\mathcal M:\mathcal A_M\to\mathcal A_M$ is continuous in $\mathcal{D}_{X,\rho}^{\alpha_0,\beta_0}([0,T_0],\mathbb R^k)$.
\end{proof}

Finally, we can present our main conclusion.

\begin{theorem}\label{thm:3}
Let $\X=(X,\x)\in \mathcal{C}^{p\text{-var}}([0,T],\mathbb{R}^d)$ be a rough path with $2\le p<3$. Suppose Assumption~\ref{assum:5} and~\ref{assum:6} hold. Then, for every $\xi\in C^{p\text{-var}}([0,t_0],\RR^k)$, there exists a solution $(Y,Y')$ to \eqref{eq:5.1} in the sense of Definition~\ref{def:9}.
\end{theorem}

\begin{proof}
It suffices to consider $t_0=0$, since the same argument applies to $t_0>0$ with the prescribed history $\xi$ kept fixed.
Choose $T_0>0$ sufficiently small. 
On the interval $[0, T_0]$, Proposition~\ref{prop:5}, Proposition~\ref{prop:6} and Proposition~\ref{prop:7} together with Schauder's fixed point theorem yield a local solution. 
Repeating the argument on successive subintervals, we can extend the solution to the whole interval $[0,T]$.
\end{proof}

\begin{remark}
The result also applies to genuinely path-dependent coefficients that need not be Fr\'echet differentiable with respect to the entire path. 
Examples include the regularized running maximum and integral dependent functionals discussed in~\cite[Example 4.4]{An23}.
Under the regularity conditions specified therein, these functionals satisfy the required Dupire assumptions. 
Consequently, the existence result extends to such coefficients throughout the full range $2\le p<3$, without imposing any additional regularity assumptions.
\end{remark}

\vskip 0.2in

\noindent
{\bf Acknowledgments.} This work is supported by the National Natural Science Foundation of China (12571019), the Natural Science Foundation of Gansu Province (25JRRA644) and Innovative Fundamental Research Group Project of Gansu Province (23JRRA684).

\noindent
{\bf Declaration of interests.} The authors have no conflicts of interest to disclose.

\noindent
{\bf Data availability.} Data sharing is not applicable as no new data were created or analyzed.

\end{document}